\documentclass[11pt,reqno]{amsart}
\usepackage{amsmath, amssymb, amsthm}
\usepackage{url}
\usepackage{mathrsfs}
\usepackage[ansinew]{inputenc}
\usepackage[breaklinks]{hyperref}
\usepackage[headheight=110pt,top=1in, bottom=.9in, left=0.8in, right=0.8in]{geometry}

\allowdisplaybreaks
\usepackage{color}

\definecolor{blue}{rgb}{0,0,1}
\definecolor{red}{rgb}{1,0,0}
\definecolor{green}{rgb}{0,.6,.2}
\definecolor{purple}{rgb}{1,0,1}

 \renewcommand{\a}{\alpha}
\renewcommand{\b}{\beta}

\newcommand{\G}{\Gamma}

\renewcommand{\(}{\left\(}
\renewcommand{\)}{\right\)}
\renewcommand{\[}{\left\[}
\renewcommand{\]}{\right\]}
\newtheorem{remark}[]{Remark}
\numberwithin{equation}{section}
 \theoremstyle{plain}
\newtheorem{theorem}{Theorem}[section]
\newtheorem{lemma}[theorem]{Lemma}

\newtheorem{corollary}[theorem]{Corollary}
\newtheorem{proposition}[theorem]{Proposition}

   \makeatletter
\def\proof{\@ifnextchar[{\@oproof}{\@nproof}}
\def\@oproof[#1][#2]{\trivlist\item[\hskip\labelsep\textit{#2 Proof of\
#1.}~]\ignorespaces}
\def\@nproof{\trivlist\item[\hskip\labelsep\textit{Proof.}~]\ignorespaces}
\makeatother

\usepackage{bigints}
\usepackage{suffix}
\usepackage{mathtools}
\DeclarePairedDelimiterX\MeijerM[3]{\lparen}{\rparen}%
{\begin{smallmatrix}#1 \\ #2\end{smallmatrix}\delimsize\vert\,#3}

\newcommand\MeijerG[8][]{%
  G^{\,#2,#3}_{#4,#5}\MeijerM[#1]{#6}{#7}{#8}}

\WithSuffix\newcommand\MeijerG*[7]{%
  G^{\,#1,#2}_{#3,#4}\MeijerM*{#5}{#6}{#7}}

\begin{document}
\title[Lipschitz summation formula and Raabe's cosine transform]{Applications of the Lipschitz summation formula and A generalization of Raabe's cosine transform} 
%{\Large }

\author{Atul Dixit}
\address{Discipline of Mathematics, Indian Institute of Technology Gandhinagar, Palaj, Gandhinagar 382355, Gujarat, India} 
\email{adixit@iitgn.ac.in}

\author{Rahul Kumar}
\address{Center for Geometry and Physics, Institute for Basic Science (IBS), Pohang 37673, Republic of Korea} 
\curraddr{Department of Mathematics, The Pennsylvania State University, University Park, PA, U.S.A.}
\email{rjk6031@psu.edu}

\thanks{2020 \textit{Mathematics Subject Classification.} Primary 11M06; Secondary 44A20, 30E20.\\
\textit{Keywords and phrases.}  Lipschitz summation formula, Raabe cosine transform, Ramanujan's formula, Hurwitz zeta function, Lambert series.}
\begin{abstract}
General summation formulas have been proved to be very useful in analysis, number theory and other branches of mathematics. The Lipschitz summation formula is one of them. In this paper, we give its application by providing a new transformation formula which generalizes that of Ramanujan. Ramanujan's result, in turn, is a generalization of the modular transformation of Eisenstein series $E_k(z)$ on SL$_2(\mathbb{Z})$, where $z\to-1/z, z\in\mathbb{H}$. The proof of our result involves delicate analysis containing Cauchy Principal Value integrals.  A simpler proof of a recent result of ours with Kesarwani giving a non-modular transformation for $\sum_{n=1}^{\infty}\sigma_{2m}(n)e^{-ny}$ is also derived using the Lipschitz summation formula. In the pursuit of obtaining this transformation, we naturally encounter a new generalization of Raabe's cosine transform whose several properties are also demonstrated. As a corollary of this result, we get a generalization of Wright's asymptotic estimate for the generating function of the number of plane partitions of a positive integer $n$. 
\end{abstract}
\maketitle

\tableofcontents

\section{Introduction}\label{intro}
Let $\a, \b>0$ with $\a\b=\pi^2$ and $m\in\mathbb{Z}\backslash\{0\}$. Ramanujan's famous formula for $\zeta(2m+1)$ is given by\footnote{Ramanujan's formula is actually valid for any complex $\a, \b$ such that $\textup{Re}(\a)>0, \textup{Re}(\b)>0$ and $\a\b=\pi^2$.} \cite[p.~173, Ch. 14, Entry 21(i)]{ramnote}, \cite[p.~319-320, formula (28)]{lnb}, \cite[p.~275-276]{bcbramsecnote} 
\begin{align}\label{rameqn}
\alpha^{-m}\left\{\frac{1}{2}\zeta(2m+1)+\sum_{n=1}^\infty \frac{n^{-2m-1}}{e^{2n\alpha}-1}\right\}&=(-\beta)^m\left\{\frac{1}{2}\zeta(2m+1)+\sum_{n=1}^\infty\frac{n^{-2m-1}}{e^{2n\beta}-1}\right\}\nonumber\\
&\qquad-2^{2m}\sum_{k=0}^m\frac{(-1)^{k}B_{2k}B_{2m+2-2k}}{(2k)!(2m+2-2k)!}\alpha^{m+1-k}\beta^k,
\end{align}
where, as customary, $\zeta(s)$ denotes the Riemann zeta function and $B_n$ denotes the $n^{\textup{th}}$ Bernoulli number defined by
\begin{equation*}
	\sum_{n=0}^{\infty}\frac{B_n z^n}{n!}=\frac{z}{e^{z}-1} \hspace{5mm}(|z|<2\pi). 
\end{equation*}
The above formula has received enormous attention from several mathematicians over the years and has been rediscovered many times, for example, see \cite[Theorem 9]{guinand1944} and \cite{malurkar}. It is an impressive result, for,  it encapsulates not only the transformation formulas of the Eisenstein series on SL$_2(\mathbb{Z})$ and the corresponding ones for their Eichler integrals but also the transformation property of the logarithm of the Dedekind eta function. For a delightful historical account on it, we refer the reader to the excellent survey \cite{berndtstraubzeta}. There are several generalizations of \eqref{rameqn} in the literature, for example, \cite{bradley2002}, \cite{dg}, \cite{dgkm}, \cite{dkk1}, \cite{dixitmaji1}, \cite{komori} and \cite{mg}.
In his second notebook \cite[p.~269]{ramnote}, Ramanujan himself provided the following generalization of \eqref{rameqn}.

Let $\alpha$ and $\beta$ be two positive real numbers such that $\alpha\beta=4\pi^2$. Then for $\mathrm{Re}(s)>2$, we have 
\begin{align}\label{ramanujan gen eqn}
\alpha^{s/2}\left\{\frac{\Gamma(s)\zeta(s)}{(2\pi)^s}+\cos\left(\frac{\pi s}{2}\right)\sum_{n=1}^\infty\frac{n^{s-1}}{e^{n\alpha}-1}\right\}&=\beta^{s/2}\left\{\cos\left(\frac{\pi s}{2}\right)\frac{\Gamma(s)\zeta(s)}{(2\pi)^s}+\sum_{n=1}^\infty\frac{n^{s-1}}{e^{n\beta}-1}\right.\nonumber\\
&\left.\quad-\sin\left(\frac{\pi s}{2}\right)\mathrm{PV}\int_0^\infty\frac{x^{s-1}}{e^{2\pi x}-1}\cot\left(\frac{1}{2}\beta x\right)dx\right\},
\end{align}
where $\mathrm{PV}$ denotes the principal value integral. The above formula has been proved in \cite[p.~416]{bcbramfifthnote}. Also see \cite[Section 9]{bdgz} for a recent generalization of \eqref{ramanujan gen eqn}.

Unfortunately, Ramanujan's formula \eqref{ramanujan gen eqn} has not received as much attention as \eqref{rameqn}. But it is also a noteworthy result because it not only gives the transformation formula for the Eisenstein series on SL$_2(\mathbb{Z})$ in the special case $s=2m, m\in\mathbb{N}, m>1$, but also reveals the obstruction to modularity for other values of $s$,  which is evident due to the appearance of the integral on its right-hand side. Note that the last term involving the integral disappears for $s=2m$.

One of the goals of this paper is to derive a generalization of \eqref{ramanujan gen eqn}:
\begin{theorem}\label{ram with a}
Let $\mathrm{Re}(\alpha),\mathrm{Re}(\beta)>0$ such that $\alpha\beta=4\pi^2$. Let $0\leq a<1$. Then, for $\mathrm{Re}(s)>2$, the following transformation holds
\begin{align}\label{ram with a eqn}
&\alpha^{s/2}\left\{\frac{\Gamma(s)\zeta(s)}{(2\pi)^s}+\frac{1}{2}\sum_{n=1}^\infty n^{s-1}\left(\frac{e^{\pi is/2}}{e^{n\alpha-2\pi ia}-1}+\frac{e^{-\pi is/2}}{e^{n\alpha+2\pi ia}-1}\right)\right\}\nonumber\\
&=\beta^{s/2}\left\{\frac{\Gamma(s)}{(2\pi)^{s}}\sum_{k=1}^\infty\frac{\cos\left(\frac{\pi s}{2}+2\pi ak\right)}{k^s}+\sum_{n=1}^\infty\frac{(n-a)^{s-1}}{e^{(n-a)\beta}-1}\right.\nonumber\\
&\quad\left.-\frac{1}{2i}\mathrm{PV}\int_0^\infty x^{s-1}
\left(\frac{e^{\pi is/2}}{e^{2\pi x-2\pi ia}-1}-\frac{e^{-\pi is/2}}{e^{2\pi x+2\pi ia}-1}\right)\cot\left(\frac{1}{2}\beta x\right)dx\right\}.
\end{align}

%If $0<a<1$, the above result holds for $\mathrm{Re}(s)>1$.
\end{theorem}

The above theorem reduces to Ramanujan's formula \eqref{ramanujan gen eqn} for $a=0$. Also for $s=2m, m\in\mathbb{N}$ and $a=0$, it gives \eqref{rameqn}.

A special case of Theorem \ref{ram with a} is the new transformation given below.
\begin{corollary}\label{a half s 2m}
Let $m\in\mathbb{N}$. For $\mathrm{Re}(\alpha),\mathrm{Re}(\beta)>0$ such that $\alpha\beta=4\pi^2$, we have
\begin{align}\label{a half s 2m eqn}
\alpha^m\sum_{n=1}^\infty\frac{n^{2m-1}}{e^{n\alpha}+1}+(-\beta)^m\sum_{n=1}^\infty\frac{(n-1/2)^{2m-1}}{e^{(n-1/2)\beta}-1}=-\left\{\alpha^m-(2^{1-2m}-1)(-\beta)^m\right\}\frac{B_{2m}}{4m}.
\end{align}
\end{corollary}
Equation \eqref{a half s 2m eqn} is a ``hybrid'' analogue of the following transformation formula for the Eisenstein series over SL$_2(\mathbb{Z})$ in that the role of $n$ in the first series is played by $n-1/2$ in the second.
\begin{align*}
\alpha^m\sum_{n=1}^\infty\frac{n^{2m-1}}{e^{n\alpha}-1}-(-\beta)^m\sum_{n=1}^\infty\frac{n^{2m-1}}{e^{n\beta}-1}=\left\{\alpha^m-(-\beta)^m\right\}\frac{B_{2m}}{4m}.
\end{align*}
As an application of Corollary \ref{a half s 2m}, we obtain closed-form evaluations of two infinite series, which, to the best of our knowledge, are new.
\begin{corollary}\label{close form}
For any odd positive integer $m$ greater than $1$,
\begin{align}\label{1stcor}
\sum_{n=1}^\infty\frac{n^{2m-1}}{e^{2n\pi}+1}-\frac{1}{2^{2m-1}}\sum_{n=1}^\infty\frac{n^{2m-1}}{e^{n\pi}-1}
%=\frac{1}{2^{2m}}(2^{2m-1}-1)\frac{B_{2m}}{m}-\frac{B_{2m}}{4m}.
=-(1+2^{1-2m})\frac{B_{2m}}{4m}.
\end{align}
Moreover,
\begin{align}\label{2ndcor}
\sum_{n=1}^\infty\frac{n}{e^{2n\pi}+1}-\frac{1}{2}\sum_{n=1}^\infty\frac{n}{e^{n\pi}-1}=-\frac{3}{48}+\frac{1}{8\pi},
\end{align}
and hence, at least one of the two series is transcendental. 
\end{corollary}
%The above equation readily implies that the series on the left-hand side is a rational number for any odd positive integer $m$ greater than $1$.

First,  let $s=2m$ and $a=1/4$ in \eqref{ram with a eqn}, then let $s=2m$ and $a=3/4$ in \eqref{ram with a eqn}, and add the corresponding sides of the resulting identities. This leads to the transformation between just the infinite series which we record below in \eqref{a=1/4and 3/4 eqn}. Similarly subtracting the corresponding sides of the two resulting identities expresses a principal value integral in terms of a Lambert series, which is given in \eqref{a=1/4and 3/4 eqn1}.
\begin{corollary}\label{a=1/4and 3/4}
Let $\mathrm{Re}(\alpha),\mathrm{Re}(\beta)>0$ such that $\alpha\beta=4\pi^2$. For $m\in\mathbb{N}, m>1$,
\begin{align}\label{a=1/4and 3/4 eqn}
&\alpha^m2^{4m-1}\left\{\frac{\Gamma(2m)\zeta(2m)}{(2\pi)^{2m}}+(-1)^{m+1}\sum_{n=1}^\infty\frac{n^{2m-1}}{e^{2n\alpha}+1}\right\}\nonumber\\
&=\beta^m\left\{(-1)^{m+1}\frac{\Gamma(2m)\zeta(2m)(2^{2m-1}-1)}{(2\pi)^{2m}}+\sum_{n=1}^\infty\frac{(4n-1)^{2m-1}}{e^{(4n-1)\beta/4}-1}+\sum_{n=1}^\infty\frac{(4n-3)^{2m-1}}{e^{(4n-3)\beta/4}-1}\right\},
\end{align}
and, 
\begin{align}\label{a=1/4and 3/4 eqn1}
 \mathrm{PV}\int_0^\infty\mathrm{sech}(2\pi x)\cot\left(2\beta x\right)x^{2m-1}dx=(-1)^{m+1}4^{1-2m}\sum_{n=1}^\infty\frac{\chi(n)n^{2m-1}}{e^{n\beta}-1},
\end{align}
where $\chi(n)$ is a Dirichlet character modulo $4$ given by
\begin{align}\label{character}
\chi(n)=
\begin{cases}
	1, \qquad \mathrm{if}\ n\equiv1\pmod{4},\\
	-1, \hspace{5mm} \mathrm{if}\ n\equiv3\pmod{4},\\
	0, \qquad \mathrm{if}\ n\equiv0, 2\pmod{4}.
\end{cases}
\end{align}
\end{corollary}
We note that the series on the right-hand side of \eqref{a=1/4and 3/4 eqn1} cannot be treated using \cite[Theorem 1]{bradley2002}.

We now transition towards to the second goal of our paper. Recently, the current authors, along with Kesarwani \cite{dkk1}, extensively studied a more general Lambert series
\begin{align}\label{ls}
\sum_{n=1}^\infty\frac{n^s}{e^{ny}-1}=\sum_{n=1}^\infty\sigma_s(n)e^{-ny} \quad (s\in\mathbb{C},\ \mathrm{Re}(y)>0)
\end{align}
than the ones appearing in \eqref{rameqn}. Here $\sigma_s(n):=\sum_{d|n}d^s$ is the generalized divisor function. Among other things, they obtained \cite[Theorem 2.5]{dkk1} an explicit transformation for any $\mathrm{Re}(s)>-1$ and Re$(y)>0$, which is given next.
\begin{align}\label{maineqn}
&\sum_{n=1}^\infty  \sigma_s(n)e^{-ny}+\frac{1}{2}\left(\left(\frac{2\pi}{y}\right)^{1+s}\mathrm{cosec}\left(\frac{\pi s}{2}\right)+1\right)\zeta(-s)-\frac{1}{y}\zeta(1-s)\nonumber\\
&=\frac{2\pi}{y\sin\left(\frac{\pi s}{2}\right)}\sum_{n=1}^\infty \sigma_{s}(n)\Bigg(\frac{(2\pi n)^{-s}}{\Gamma(1-s)} {}_1F_2\left(1;\frac{1-s}{2},1-\frac{s}{2};\frac{4\pi^4n^2}{y^2} \right) -\left(\frac{2\pi}{y}\right)^{s}\cosh\left(\frac{4\pi^2n}{y}\right)\Bigg),
\end{align}
where ${}_1F_2(a;b,c;z):=\sum_{n=0}^\infty\frac{(a)_nz^n}{(b)_n(c)_nn!},\ z\in\mathbb{C},\ (a)_n=\frac{\Gamma(a+n)}{\Gamma(a)}$ is the generalized hypergeometric function.

The explicit transformations of the type \eqref{maineqn} are always desirable due to their possible applications in analytic number theory, especially in the theory of zeta functions. See the recent paper \cite{bdg} for a beautiful application of \eqref{maineqn} in the theory of $\zeta(s)$ by applying the operator $\frac{d}{ds}\big|_{s=0}$ on both sides, thereby resulting in a transformation of the Lambert series of the logarithm, that is, $\displaystyle\sum_{n=1}^\infty\frac{\log(n)}{e^{ny}-1}$.

The authors of \cite[Theorem 2.5]{dkk1} also analytically continued \eqref{maineqn} to  Re$(s)>-2m-3,\ m\in\mathbb{N}\cup\{0\}$. Then, as a special case, they not only obtained Ramanujan's formula \eqref{rameqn} and the transformation formula of the logarithm of the Dedekind eta function but also new transformations when $s$ is an even integer. For example, they established an explicit result \cite[Theorem 2.11]{dkk1} for \eqref{ls} when $s=2m,\ m>0$. We record it below in \eqref{dkka=2m eqn}. It comprises two special functions Shi$(z)$ and Chi$(z)$, known as the hyperbolic sine and cosine integrals, respectively defined by \cite[p.~150, Equation (6.2.15), (6.2.16)]{nist}
\begin{align}\label{shichi}
	\mathrm{Shi}(z):=\int_0^z\frac{\sinh(t)}{t}\ dt,\hspace{4mm}
	\mathrm{Chi}(z):=\gamma+\log(z)+\int_0^z\frac{\cosh(t)-1}{t}\ dt,
\end{align}
%\begin{align}\label{shichi}
%\mathrm{Shi}(z):=&\int_0^z\frac{\sinh(t)}{t}\ dt,\nonumber\\
%\mathrm{Chi}(z):=&\gamma+\log(z)+\int_0^z\frac{\cosh(t)-1}{t}\ dt,
%\end{align}
where $\gamma$ is Euler's constant. Let $m\in\mathbb{N}$. Then for $\mathrm{Re}(y)>0$, we have  \cite[Theorem 2.11]{dkk1}
{\allowdisplaybreaks\begin{align}\label{dkka=2m eqn}
\sum_{n=1}^\infty \sigma_{2m}(n)e^{-ny}-\frac{(2m)!}{y^{2m+1}}\zeta(2m+1)+\frac{B_{2m}}{2my}&=(-1)^m\frac{2}{\pi}\left(\frac{2\pi}{y}\right)^{2m+1}\sum_{n=1}^\infty\sigma_{2m}(n)\Bigg\{\sinh\left(\frac{4\pi^2n}{y}\right)\mathrm{Shi}\left(\frac{4\pi^2n}{y}\right)\nonumber\\
&\quad-\cosh\left(\frac{4\pi^2n}{y}\right)\mathrm{Chi}\left(\frac{4\pi^2n}{y}\right)+\sum_{j=1}^m(2j-1)!\left(\frac{4\pi^2n}{y}\right)^{-2j}\Bigg\}.
\end{align}}
The modular transformation for $\sum_{n=1}^{\infty}\sigma_{2m+1}(n)e^{-ny}$ transforms it into \begin{equation}\label{diff}
	\sum_{n=1}^\infty\sigma_{2m+1}(n)e^{-4\pi^2n/y}=-\sum_{n=1}^\infty\sigma_{2m+1}(n)\Bigg\{\sinh\left(\frac{4\pi^2n}{y}\right)-\cosh\left(\frac{4\pi^2n}{y}\right)\Bigg\}.
\end{equation}
In view of this, it is important to note that while going from $s=2m+1$ to $s=2m$ in $\sum_{n=1}^\infty  \sigma_s(n)e^{-ny}$, the expression $\sinh\left(\frac{4\pi^2n}{y}\right)-\cosh\left(\frac{4\pi^2n}{y}\right)$ in \eqref{diff} is to be replaced by the corresponding one on the right-hand side of \eqref{dkka=2m eqn}. (Note that the finite sum $\sum_{j=1}^m(2j-1)!\left(\frac{4\pi^2n}{y}\right)^{-2j}$ in the summand of the series on the right-hand side of \eqref{diff} is essential for its convergence; for details, see the proof of \eqref{dkka=2m eqn} in Section \ref{appl}.)

Equation \eqref{dkka=2m eqn} readily gives the following asymptotic estimate for $\sum_{n=1}^\infty \sigma_{2m}(n)e^{-ny}$:
\begin{corollary}\label{asym for general m}
Let $m\in\mathbb{N}$. As $y\to0$ in $|\arg(y)|<\pi/2$, 
\begin{align}\label{asym for general m eqn}
\sum_{n=1}^\infty \sigma_{2m}(n)e^{-ny}&=\frac{(2m)!}{y^{2m+1}}\zeta(2m-1)-\frac{B_{2m}}{2my}-\frac{2(-1)^m}{\pi(2\pi)^{2m-1}}\sum_{j=1}^{r+1}\frac{\Gamma(2m+2j)\zeta(2m+2j)\zeta(2j)}{(2\pi)^{4j}}y^{2j-1}+O\left(y^{2r+3}\right).
\end{align}
\end{corollary}
%In view of this, it is important to note that for $s=2m$, the part of the summand of  the series on the right-hand side of \eqref{diff} involving the difference of hyperbolic functions is replaced in the series on the right-hand side of \eqref{dkka=2m eqn} by the corresponding one involving not only hyperbolic functions but also hyperbolic sine and cosine integrals.
The case $m=1$ of the series on the left-hand side of \eqref{dkka=2m eqn} (or \eqref{asym for general m}) has the following interesting connection with the generation function for plane partitions studied by MacMahon \cite[p.~184]{andrews}:
\begin{align}\label{plain partition}
\sum_{n=1}^\infty \sigma_{2}(n)e^{-ny}=x\frac{d}{dx}\log(F(x)),
\end{align}
where $\displaystyle F(x):=\prod_{n=1}^\infty\frac{1}{(1-x^n)^n}$, with $y=\log(1/x)$, where $|x|<1$. In his work on finding the asymptotic estimate of $q(n)$, the number of plane partitions of a positive integer $n$, Wright \cite[Lemma 1]{wright} first found the asymptotic estimate of $F(x)$ as $x\to1^-$. His result on $F(x)$ readily follows from our Corollary \ref{asym for general m} and is rephrased in the following corollary.
\begin{corollary}\label{wright asym}
As $x\to1^-$, we have
\begin{align}\label{wright asym eqn}
F(x)=e^c(\log x)^{1/12}\exp\left(\frac{\zeta(3)}{\log^2x}\right)\exp\left(-\sum_{j=1}^{r+1}\delta_j(\log x)^{2j}\right)\left(1+O_r\left((\log x)^{2r+4}\right)\right),
\end{align}
where, $c$ is a constant, and
\begin{align}\label{delta j}
\delta_j:=\frac{\Gamma(2j+2)\zeta(2j+2)\zeta(2j)}{2\pi^2j(2\pi)^{4j}}.
\end{align}
\end{corollary}

It is important to note that Wright obtained the above result through a long calculation (see \cite[pp. 180-184]{wright} whereas it is a trivial consequence of our Corollary \ref{asym for general m} as shown in Section \ref{asymptotic}. On the other hand, the advantage of his method is that it gives a representation of the constant $c$ in terms of an integral, namely, $c=\displaystyle2\int_{0}^{\infty}\frac{y\log(y)}{e^{2\pi y}-1}\, dy$.

The Lambert series \eqref{ls}, whose special cases were considered in \eqref{rameqn} and \eqref{dkka=2m eqn}, has been studied by many mathematicians over the years. For a detailed survey, see \cite{dkk1}. One of the earliest mathematicians to study it was Wigert, who wrote several papers on this subject. In \cite{wigert}, Wigert examined the Lambert series when $0<s<1$. Later, Kuylenstierna \cite{ku} provided a simple proof of Wigert's result using double zeta $\zeta_2(s,\tau):=\sum_{m,n=0}^\infty\frac{1}{(m+n\tau)^s}$, Re$(s)>2$, $\tau\in\mathbb{C}\backslash(-\infty,0]$. However, both of them were interested only in the asymptotics of the series in \eqref{ls}, not in explicit transformations. In his work, Kuylenstierna essentially uses the Lipschitz summation formula \cite{lipschpaper}, namely, 
%We note down  it below. 
%Theorem \ref{a=2mcase} is a straightforward application of this summation formula.
$0\leq a<1$, $\mathrm{Re}(s)>1$ and $\tau\in\mathbb{H}$,
\begin{align}\label{lipscheqn}
\sum_{n=1}^\infty\frac{e^{2\pi i\tau(n-a)}}{(n-a)^{1-s}}=\frac{\Gamma(s)}{(-2\pi i)^s}\sum_{k\in\mathbb{Z}}\frac{e^{2\pi iak}}{(k+\tau)^s}.
\end{align}
The Lipschitz summation formula has several nice applications and generalizations, for example,  see \cite{berndtacta, knopprobins, pw}. \eqref{lipscheqn} is usually proved using Poisson summation formula, for example, see \cite[p.~77--79]{rad}. For other proofs, one can look at the paper of V\'{a}gi \cite{vagi}.

It does not seem to be easy to get \eqref{dkka=2m eqn} as a special case of Ramanujan's formula \eqref{ramanujan gen eqn} or our generalization \eqref{ram with a eqn}, because one has to transform the Lambert series and the principal value integral on the right-hand side of \eqref{ramanujan gen eqn} into the series in \eqref{dkka=2m eqn} involving the special functions Shi$(z)$ and Chi$(z)$.

In this paper, we prove Theorem \ref{ram with a} and  \eqref{dkka=2m eqn} using the Lipschitz summation formula \eqref{lipscheqn}.
%, which means we are getting transformation formulas for the Lambert series \eqref{ls} for $a=2m$ and $a=s-1$ from a single summation formula. 
The proof of \eqref{dkka=2m eqn} through this approach involves a nice generalization of the following identity \cite[Theorem 2.2]{dgkm} 
\begin{align}\label{dgkmresult}
\sum_{n=1}^\infty\int_0^\infty\frac{t\cos(t)}{t^2+n^2u^2}\ dt=\frac{1}{2}\left\{\log\left(\frac{u}{2\pi}\right)-\frac{1}{2}\left(\psi\left(\frac{iu}{2\pi}\right)+\psi\left(-\frac{iu}{2\pi}\right)\right)\right\},
\end{align}
where  $\textup{Re}(u)>0$, and $\psi(z):=\G'(z)/\G(z)$ is the digamma function. In  \cite[Theorem 2.4]{dgkm}, the above identity was employed to obtain a two-parameter generalization of \eqref{rameqn}. Various applications of \eqref{dgkmresult} can be found in \cite{dgk, dkk1}.

Observe that the summand of the left-hand side of \eqref{dgkmresult} is the Raabe cosine transform defined for Re$(w)>0$ and $y>0$ by \cite[p.~144]{htf2}
\begin{align*}
\mathfrak{R}(y,w):=\int_0^\infty\frac{t\cos(yt)}{t^2+w^2}\ dt.
\end{align*}

Before stating the generalization of \eqref{dgkmresult}, that  is sought for, we first introduce  a new generalization of Raabe's cosine transform, valid for $\textup{Re}(w)>0, \textup{Re}(z)>0$ and $y>0$, by 
\begin{align}\label{genraabe}
\mathfrak{R}_z(y,w):=\frac{1}{2}\Gamma(2z+1)\int_0^\infty\left(\frac{1}{(t-iw)^{2z+1}}+\frac{1}{(t+iw)^{2z+1}}\right)\cos(yt)\ dt.
\end{align} 

It is easy to see that $\mathfrak{R}_0(y,w)=\mathfrak{R}(y,w)$. Also for $w>0$, $\mathfrak{R}_z(y,w)$ satisfies a nice identity, namely,
\begin{equation}\label{symmetry}
w^{2z}\mathfrak{R}_z(y,w)=y^{2z}\mathfrak{R}_z(w, y),
\end{equation}
which is easily seen by making the change of variable $t=xw/y$ in \eqref{genraabe}.

Our first result on $\mathfrak{R}_z(y,w)$ gives a closed-form evaluation of an infinite series containing $\mathfrak{R}_z(y,w)$.
% This is an important ingredient in proving Theorem \ref{a=2mcase}.
\begin{theorem}\label{nagoyagen}
Let $\zeta(z,a)$ be the Hurwitz zeta function. For $\mathrm{Re}(w)>0$ and $\mathrm{Re}(z)>0$, we have
{\allowdisplaybreaks\begin{align}\label{nagoyageneqn}
\frac{2}{\Gamma(2z+1)}\sum_{n=1}^\infty \mathfrak{R}_z(2\pi n,w)&=\sum_{n=1}^\infty\int_0^\infty\left(\frac{1}{(v-iw)^{2z+1}}+\frac{1}{(v+iw)^{2z+1}}\right)\cos(2\pi nv)\ dv\nonumber\\
&=\frac{1}{2}\left\{\zeta(1+2z,iw)+\zeta(1+2z,-iw)-\frac{\cos(\pi z)}{zw^{2z}}\right\}.
\end{align}}
\end{theorem}

Note that this result is not straightforward to obtain as one cannot interchange the order of the summation and integration as doing so leads to a divergent integral. The primary tool to prove this result is Guinand's generalization of Poisson's summation formula \cite[Theorem 1]{apg1}; see Theorem \ref{Guind-pois}. 

The generalized Raabe cosine transform $\mathfrak{R}_z(y,w)$ itself can be evaluated in terms of exponential integral functions and incomplete gamma functions which are not so popular. However, the beauty of Theorem \ref{nagoyagen} is that the infinite sum of $\mathfrak{R}_z(y,w)$ can be evaluated in terms of the well-known functions such as the Hurwitz zeta function $\zeta(z,a)$ and $\cos(z)$.

An immediate consequence of Theorem \ref{nagoyagen} is
\begin{corollary}\label{cornagoya}
Equation \eqref{dgkmresult} holds true.
\end{corollary}

Our next result gives an evaluation of a double integral which is imperative to prove Theorem \ref{nagoyagen}. 
\begin{theorem}\label{intiszero}
Let $\mathrm{Re}(w)>0$ and  $\mathrm{Re}(z)>0$. Then
\begin{align}\label{intiszeroeqn}
\int_0^\infty\int_0^\infty\left(\frac{1}{(t-iw)^{2z+1}}+\frac{1}{(t+iw)^{2z+1}}\right)\cos(2\pi vt)\ dtdv=-\frac{1}{2w^{2z+1}}\sin(\pi z).
\end{align}
Equivalently, in the notation of \eqref{genraabe},
\begin{align}
\int_0^\infty\mathfrak{R}_z(2\pi v,w)\ dv=-\frac{1}{4w^{2z+1}}\Gamma(2z+1)\sin(\pi z).\nonumber
\end{align}
\end{theorem}

It is effortless to see that for $z\in\mathbb{N}\cup\{0\}$, the above integral evaluates to zero. The particular case $z=0$ is already obtained in \cite[Lemma 3.4]{dgkm}.

We now provide an new equivalent representation for $\mathfrak{R}_m(1,w)$, where $m\in\mathbb{N}\cup\{0\}$. This representation appears in the transformation of $\sum_{n=1}^{\infty}\sigma_{2m}(n)e^{-ny}$ given in \eqref{dkka=2m eqn}. 
\begin{theorem}\label{minusonelemma}
Let $\mathrm{Shi}(z)$, $\mathrm{Chi}(z)$ and $\mathfrak{R}_z(y,w)$ be defined in \eqref{shichi} and \eqref{genraabe} respectively. Let $m\in\mathbb{N}\cup\{0\}$ and $\mathrm{Re}(w)>0$. Then
\begin{align}\label{minusone}
\mathfrak{R}_m(1,w)&=\frac{(2m)!}{2}\int_0^\infty\left(\frac{1}{(t-iw)^{2m+1}}+\frac{1}{(t+iw)^{2m+1}}\right)\cos(t)\ dt\nonumber\\
&=(-1)^m\left(\sinh(w)\mathrm{Shi}(w)-\cosh(w)\mathrm{Chi}(w)+\sum_{j=1}^m(2j-1)!w^{-2j}\right).
\end{align}
\end{theorem}

The special case $m=0$ of this result was derived in \cite[Lemma 9.1]{dkk1}.
 
As mentioned earlier, we provide a new proof of \eqref{dkka=2m eqn} in this paper. It is done by employing the Lipschitz summation formula and Theorems \ref{nagoyagen} and \ref{minusonelemma}. Deriving it this way is simpler than obtaining it as a special case of \eqref{maineqn}. The latter was done in \cite[Section 9]{dkk1}.

%One more crucial result to prove Theorem \ref{dkka=2m} and as well as Theorem \ref{ramanujan thm} is stated in the next theorem. This is another new transformation for the series $\displaystyle\sum_{n=1}^\infty \sigma_{s}(n)e^{-ny}$.
%\begin{theorem}\label{a=2mcase}
%Let $\mathrm{Re}(s)>1$. Then for $\mathrm{Re}(y)>0$, the following transformation holds
%\begin{align}\label{a=2mcaseeqn}
%&\sum_{n=1}^\infty \sigma_{s}(n)e^{-ny}=\frac{\Gamma(s)}{y^{s}}\sum_{k=-\infty}^\infty \zeta\left(s,1+\frac{2\pi ik}{y}\right).
%%\left\{\zeta(2m+1)+\sum_{k=1}^\infty\left(\zeta\left(1+2m,1-\frac{2\pi ik}{y}\right)+\zeta\left(1+2m,1+\frac{2\pi ik}{y}\right)\right)\right\}.
%\end{align}
%\end{theorem}
%
%This result is a easy implication of the Lipschitz summation formula \eqref{lipscheqn}. For more details, see Section \ref{equivalent}. 

This paper is organised as follows. In Section \ref{grct}, the proofs of  Theorems \ref{nagoyagen}, \ref{intiszero} and \ref{minusonelemma} are given.  Sections  \ref{section ram thm} and \ref{appl} are devoted to proving Theorem \ref{ram with a} and \eqref{dkka=2m eqn} respectively.

%The novelty of our proof is that it not only needs properties of a new special function

\section{The generalized Raabe cosine transform $\mathfrak{R}_z(y,w)$}\label{grct}

This section is devoted to obtaining the results associated with $\mathfrak{R}_z(y,w)$ and which are crucial to proving \eqref{dkka=2m eqn}. The first result below gives the asymptotic expansion of $\mathfrak{R}_z(y,w)$ as $y\to\infty$.
\begin{lemma}\label{genraabeasyinfty}
Let $\mathfrak{R}_z(y,w)$ be defined in \eqref{genraabe}. Let $\mathrm{Re}(w)>0$ and $\mathrm{Re}(z)>0$. Then as $y\to\infty$,
\begin{align}
\mathfrak{R}_z(y,w)\sim-\frac{\cos(\pi z)}{w^{2z}}\sum_{n=1}^\infty \frac{\Gamma(2z+2n)}{(yw)^{2n}}.\nonumber
\end{align}
\end{lemma}

\begin{proof}
We use the analogue of Watson's lemma for Laplace transform in the setting of Fourier transforms \cite{olver1974}, \cite[Equations (1.3), (1.4)]{dainaylor}. It states that if the form of $h(t)$ near $t=0$ is given as a series of algebraic powers, that is,
\begin{equation}\label{algh}
h(t)\sim\sum_{n=0}^{\infty}b_nt^{n+\lambda-1}
\end{equation} 
as $t\to 0^{+}$, then under certain restrictions on $h$ (see \cite{olver1974}, \cite[Section 2]{dainaylor} for the same),
\begin{align}\label{olvthm}
\int_{0}^{\infty}e^{ist}h(t)\, \mathrm{d}t\sim \sum_{n=0}^{\infty}b_ne^{i(n+\lambda)\pi/2}\Gamma(n+\lambda)s^{-n-\lambda}
\end{align}
as $s\to\infty$.

Let $$h(t):=\displaystyle\frac{1}{(t-iw)^{2z+1}}+\frac{1}{(t+iw)^{2z+1}}.$$ 
Then, near $t=0$, it is easy to see that
\begin{align}
h(t)&=
%&=(-iw)^{-(2z+1)}\left(1-\frac{t}{iw}\right)^{-(2z+1)}+(iw)^{-(2z+1)}\left(1+\frac{t}{iw}\right)^{-(2z+1)}\nonumber\\
(-iw)^{-(2z+1)}\sum_{n=0}^\infty\frac{(2z+1)_n}{n!}\left(\frac{t}{iw}\right)^n+(iw)^{-(2z+1)}\sum_{n=0}^\infty\frac{(2z+1)_n}{n!}\left(-\frac{t}{iw}\right)^n\nonumber\\
&=2w^{-(2z+1)}\sum_{n=0}^\infty\frac{(2z+1)_n}{n!w^n}\sin\left(\frac{\pi n}{2}-\pi z\right)t^n.\nonumber
\end{align}
Therefore, it is clear that our function $h(t)$ satisfies \eqref{algh} with $\lambda=1$ and 
\begin{align}\label{bn}
b(n)=\frac{2w^{-(2z+1)}(2z+1)_n}{n!w^n}\sin\left(\frac{\pi n}{2}-\pi z\right).
\end{align} 
From \eqref{olvthm} and \eqref{bn},  as $y\to\infty$,
\begin{align}\label{plusv}
\int_0^\infty \left(\frac{1}{(t-iw)^{2z+1}}+\frac{1}{(t+iw)^{2z+1}}\right) e^{ iyt}\ dt\sim\sum_{n=0}^\infty b_ne^{i(n+1)\frac{\pi}{2}}\Gamma(n+1)y^{-n-1},
\end{align}
where $b(n)$ is given in \eqref{bn}. Similarly, as $y\to\infty$,
\begin{align}\label{minusv}
\int_0^\infty \left(\frac{1}{(t-iw)^{2z+1}}+\frac{1}{(t+iw)^{2z+1}}\right) e^{-iyt}\ dt\sim\sum_{n=0}^\infty b_ne^{i(n+1)\frac{\pi}{2}}\Gamma(n+1)(-y)^{-n-1}.
\end{align}
From \eqref{plusv} and \eqref{minusv}, we see that as $y\to\infty$,
\begin{align}\label{fullk}
&\int_0^\infty \left(\frac{1}{(t-iw)^{2z+1}}+\frac{1}{(t+iw)^{2z+1}}\right)\cos(yt)\ dt\nonumber\\
&\sim w^{-(2z+1)}\sum_{n=0}^\infty \frac{(2z+1)_n}{w^ny^{n+1}}e^{i(n+1)\frac{\pi}{2}}\sin\left(\frac{\pi n}{2}-\pi z\right)(1+(-1)^{-n-1})\nonumber\\
&=2w^{-(2z+1)}\sum_{n=1}^\infty\frac{(2z+1)_{(2n-1)}}{w^{2n-1}y^{2n}}e^{n\pi i}\sin\left(\frac{\pi (2n-1)}{2}-\pi z\right)\nonumber\\
&=-\frac{2w^{-2z}\cos(\pi z)}{\Gamma(2z+1)}\sum_{n=1}^\infty\frac{\Gamma(2z+2n)}{(yw)^{2n}}.
\end{align}
Lemma \ref{genraabeasyinfty}  follows upon multiplying both sides of \eqref{fullk} by $\frac{1}{2}\Gamma(2z+1)$ and then using the definition of $\mathfrak{R}_z(y,w)$ from \eqref{genraabe}.
\end{proof}
\begin{remark}
The special case $z=0$ of Lemma \ref{genraabeasyinfty} was obtained in \cite[Lemma 3.3]{dgkm}.
\end{remark}
%\begin{corollary}
%Let $\mathrm{Re}(w)>0$. Then as $y\to\infty$,
%\begin{align}\label{genraabeasyinftyeqn1}
%R(y,w)\sim-\sum_{n=1}^\infty \frac{\Gamma(2n)}{(yw)^{2n}}.
%\end{align}
%\end{corollary}
Our next task is to evaluate the double integral in \eqref{intiszeroeqn}.

\begin{proof}[Theorem \textup{\ref{intiszero}}][]
Note that double integral in \eqref{intiszeroeqn} is not absolutely convergent which means we cannot interchange the order of integration. Securing convergence of the integral over $v$ near $v=0$ is straightforward. Along with this,  Lemma \ref{genraabeasyinfty} implies that the double integral in \eqref{intiszeroeqn} is convergent. 

We first evaluate a more general integral by introducing the exponential factor $e^{-\frac{v^2}{N}}$ inside the integrand and then take limit $N\to\infty$. Let $N$ be a positive integer and consider the integral
\begin{align}\label{IN}
I(w,z,N):&=\int_0^\infty\int_0^\infty e^{-\frac{v^2}{N}}\left(\frac{1}{(t-iw)^{2z+1}}+\frac{1}{(t+iw)^{2z+1}}\right)\cos(2\pi vt)\ dtdv \quad(\mathrm{Re}(w)>0,\ \mathrm{Re}(z)>0).
\end{align}
By invoking Fubini's theorem we can interchange the order of the summation and integration in the above equation to see that
\begin{align}\label{bid}
I(w,z,N)&=\int_0^\infty\left(\frac{1}{(t-iw)^{2z+1}}+\frac{1}{(t+iw)^{2z+1}}\right)\int_0^\infty e^{-\frac{v^2}{N}}\cos(2\pi vt)\ dvdt\nonumber\\
&=\frac{\sqrt{\pi N}}{2}\int_0^\infty e^{-N\pi^2t^2}\left(\frac{1}{(t-iw)^{2z+1}}+\frac{1}{(t+iw)^{2z+1}}\right) dt,
\end{align}
where we used the fact that $e^{-v^2/N}$ is self-reciprocal (up to some factor) with respect to the cosine kernel (See \cite[p.~488, Formula 3.896.4]{gr}). Next invoke the identity \cite[p.~88, Section 2.5.5]{erd}
\begin{align}
(1-\sqrt{\xi})^{-2s}+(1+\sqrt{\xi})^{-2s}=2\ {}_2F_1\left(s,s+\frac{1}{2};\frac{1}{2};\xi\right),\nonumber
\end{align}
with $\xi=-w^2/t^2$ and $s=z+1/2$ in \eqref{bid} to deduce that
\begin{align}\label{before prud}
I(w,z,N)&=\sqrt{\pi N}\int_0^\infty e^{-N\pi^2t^2}t^{-2z-1}{}_2F_1\left(z+\frac{1}{2},z+1;\frac{1}{2};-\frac{w^2}{t^2}\right)dt\nonumber\\
&=\frac{\sqrt{\pi N}}{2}\int_0^\infty e^{-N\pi^2/x}x^{z-1}{}_2F_1\left(z+\frac{1}{2},z+1;\frac{1}{2};-w^2x\right)dx,
\end{align}
where we made the change of variable $t=1/\sqrt{x}$. From \cite[p~319, Formula 2.21.2.6]{prud}, for $\mathrm{Re}(p)>0, \mathrm{Re}(a-\alpha)>0,\ \mathrm{Re}(b-\alpha)>0$ and $|\arg (\omega)|<\pi$, we have
\begin{align*}
\int_0^\infty x^{\alpha-1}e^{-p/x}{}_2F_1(a,b;c;-\omega x)dx&=\omega^{-\alpha}\frac{\Gamma(c)\Gamma(\alpha)\Gamma(a-\alpha)\Gamma(b-\alpha)}{\Gamma(a)\Gamma(b)\Gamma(c-\alpha)}{}_2F_2(a-\alpha,b-\alpha;1-\alpha,c-\alpha;\omega p)\nonumber\\
&\quad+p^\alpha\Gamma(-\alpha){}_2F_2(a,b;c,\alpha+1;\omega p).
\end{align*}
Let $p=N\pi^2,\ a=z+1/2,\ b=z+1,\ c=1/2,\ \alpha=z$ and $\omega=w^2$ in the above integral evaluation, use the reflection formula for the gamma function $\Gamma(1/2+s)\Gamma(1/2-s)=\pi/\cos(\pi s)$  and substitute the resultant in \eqref{before prud} so that for $|\arg (w)|<\pi/2$,
\begin{align}\label{eval in N}
I(w,z,N)&=\frac{\sqrt{\pi N}}{2}\left\{\frac{\cos(\pi z)}{zw^{2z}}{}_2F_2\left(\frac{1}{2},1;1-z,\frac{1}{2}-z;N\pi^2w^2\right)+(N\pi^2)^z\Gamma(-z){}_1F_1\left(z+\frac{1}{2};\frac{1}{2};N\pi^2w^2\right)\right\}.
\end{align}
We now wish to take limit $N\to\infty$ on both sides of the above equation. To that end, we need to find the behavior of the functions on the right-hand side as $N\to\infty$. The following asymptotic is given by Kim \cite{kim}: as $x\to\infty$ in $-\frac{3\pi}{2}<\mathrm{\arg}(x)<\frac{\pi}{2}$, for $\alpha\neq\mathbb{Z}\cup\{0\}$,
\begin{align*}
{}_2F_1(1,\alpha;\rho_1,\rho_2;x)\sim\frac{\Gamma(\rho_1)\Gamma(\rho_2)}{\Gamma(\alpha)}\left(K_{22}(x)+L_{22}(-x)\right),
\end{align*}
where, with $\nu=1+\alpha-\rho_1-\rho_2$,
\begin{align*}
K_{22}(x)=x^\nu e^x{}_2F_0\left(\rho_1-\alpha,\rho_2-\alpha;-;\frac{1}{x}\right),
\end{align*}
and 
\begin{align*}
L_{22}(x)&=x^{-1}\frac{\Gamma(\alpha-1)}{\Gamma(\rho_1-1)\Gamma(\rho_2-1)}{}_3F_1\left(1,2-\rho_1,2-\rho_2;2-\alpha;\frac{1}{x}\right)\nonumber\\
&\quad+x^{-\alpha}\frac{\Gamma(\alpha)\Gamma(1-\alpha)}{\Gamma(\rho_1-\alpha)\Gamma(\rho_1-\alpha)}{}_2F_0\left(1+\alpha-\rho_1,1+\alpha-\rho_2;-;\frac{1}{x}\right).
\end{align*}
We let $\alpha=1/2,\ \rho_1=1-z,\ \rho_2=1/2-z$ and  $x=N\pi^2w^2$ in the above expression to get, for $-\frac{3\pi}{4}<\mathrm{\arg}(w)<\frac{\pi}{4}$, 
\begin{align}\label{asy 2f2}
{}_2F_2\left(\frac{1}{2},1;1-z,\frac{1}{2}-z;N\pi^2w^2\right)&\sim\frac{\Gamma(1/2-z)\Gamma(1-z)}{\sqrt{\pi}}\left\{(N\pi^2w^2)^{2z}e^{N\pi^2w^2}{}_2F_0\left(\frac{1}{2}-z,-z;-;\frac{1}{N\pi^2w^2}\right)\right.\nonumber\\
&\left.\quad+\frac{2}{N\pi^{3/2}w^2\Gamma(-1/2-z)\Gamma(-z)}{}_3F_1\left(1,1+z,\frac{3}{2}+z;\frac{3}{2};-\frac{1}{N\pi^2w^2}\right)\right.\nonumber\\
&\left.\quad+\frac{\pi(-N\pi^2w^2)^{-1/2}}{\Gamma(1/2-z)\Gamma(-z)}{}_2F_0\left(\frac{1}{2}+z,1+z;-;-\frac{1}{N\pi^2w^2}\right)\right\}
\end{align}
as $N\to\infty$. Also, from \cite[p.~189, Exercise 7.7]{temme},
\begin{align*}
{}_1F_1(a;c;x)\sim\frac{e^xx^{a-c}\Gamma(c)}{\Gamma(a)}\sum_{n=0}^\infty\frac{(c-a)_n(1-a)_n}{n!}x^{-n}+\frac{e^{-\pi ia}x^{-a}}{\Gamma(c-a)}\sum_{n=0}^\infty\frac{(a)_n(1+a-c)_n}{n!}(-x)^{-n}, \quad x\to\infty,
\end{align*}
where $-\frac{3\pi}{2}<\mathrm{arg}(x)<\frac{\pi}{2}$. Upon letting $a=z+1/2,\ c=1/2$ and $x=N\pi^2w^2$ in the above formula and using the series definition of ${}_2F_0$, for $-\frac{3\pi}{4}<\mathrm{arg}(w)<\frac{\pi}{4}$, we see that
\begin{align}\label{asy 1f1}
{}_1F_1\left(z+\frac{1}{2};\frac{1}{2};N\pi^2w^2\right)&\sim e^{N\pi^2w^2}\frac{\sqrt{\pi}(N\pi^2w^2)^z}{\Gamma(z+1/2)}{}_2F_0\left(-z,\frac{1}{2}-z;-;\frac{1}{N\pi^2w^2}\right)\nonumber\\
&\quad+e^{-\pi i(z+1/2)}\frac{\sqrt{\pi}(N\pi^2w^2)^{-(z+1/2)}}{\Gamma(-z)}{}_2F_0\left(z+\frac{1}{2},1+z;-;-\frac{1}{N\pi^2w^2}\right)
\end{align}
as $N\to\infty$. Substitute  \eqref{asy 2f2} and \eqref{asy 1f1} in \eqref{eval in N} and observe that the terms involving ${}_2F_0\left(-z,\frac{1}{2}-z;-;\frac{1}{N\pi^2w^2}\right)$ cancel each other out. Also note that ${}_pF_q(a_1,\cdots, a_p;b_1,\cdots,b_q;1/N)=1+O(1/N)$, as $N\to\infty$. Hence, for $-\frac{\pi}{2}<\arg (w)<\frac{\pi}{4}$, as $N\to\infty$, 
\begin{align}
I(w,z,N)&=\frac{\sqrt{\pi}}{2}\left\{\frac{2^{2z}\Gamma(1-2z)\cos(\pi z)}{zw^{2z}}\left[\frac{1}{\sqrt{N}}\frac{2^{-2z-1}}{\pi^2w^2\Gamma(-1-2z)}\left(1+O\left(\frac{1}{N}\right)\right)\right.\right.\nonumber\\
&\left.\left.\quad+\frac{i2^{-2z-1}}{\sqrt{\pi}w\Gamma(-2z)}\left(1+O\left(\frac{1}{N}\right)\right)\right]-i\frac{1}{\sqrt{ \pi}} e^{-\pi iz}w^{-2z-1}\left(1+O\left(\frac{1}{N}\right)\right)\right\}.\nonumber
\end{align}
We next let $N\to\infty$ on the both sides of the above equation. By using the dominated convergence theorem, we can take the limit $N\to\infty$ inside the integral sign in \eqref{IN}. Thus, 
\begin{align}\label{lasser region}
\int_0^\infty\int_0^\infty \left(\frac{1}{(t-iw)^{2z+1}}+\frac{1}{(t+iw)^{2z+1}}\right)\cos(2\pi vt)\ dtdv&=-\frac{i}{2w^{2z+1}}\left\{e^{-\pi iz}-\cos(\pi z)\right\}\nonumber\\
&=-\frac{i}{2w^{2z+1}}\left\{e^{-\pi iz}-\frac{e^{i\pi z}+e^{-i\pi z}}{2}\right\}\nonumber\\
&=\frac{i}{2w^{2z+1}}\left\{\frac{e^{i\pi z}-e^{-i\pi z}}{2}\right\},
\end{align}
which proves our theorem for $-\frac{\pi}{2}<\mathrm{\arg}(w)<\frac{\pi}{4}$. We next prove  the result in the remaining region $\frac{\pi}{4}\leq\mathrm{\arg}(w)<\frac{\pi}{2}$.

 By invoking the asymptotic \cite[p.~411, Formula 16.11.7]{nist} twice, for $\frac{\pi}{4}\leq\mathrm{\arg}(w)<\frac{\pi}{2}$, as $N\to\infty$,
\begin{align}\label{2f2nist}
{}_2F_2\left(\frac{1}{2},1;1-z,\frac{1}{2}-z;N\pi^2w^2\right)&\sim\frac{\Gamma(1-z)\Gamma\left(\frac{1}{2}-z\right)}{\Gamma(1/2)}\left\{\frac{e^{\pi i/2}}{\sqrt{N}\pi w}\sum_{k=0}^\infty\frac{(-1)^k\Gamma\left(\frac{1}{2}+k\right)\Gamma\left(\frac{1}{2}-k\right)}{k!\Gamma\left(\frac{1}{2}-z-k\right)\Gamma\left(-z-k\right)}\left(N\pi^2w^2e^{-\pi i}\right)^{-k}\right.\nonumber\\
&\quad\left.+\frac{e^{\pi i}}{N\pi^2 w^2}\sum_{k=0}^\infty\frac{(-1)^k\Gamma\left(-\frac{1}{2}-k\right)}{k!\Gamma\left(-z-k\right)\Gamma\left(-\frac{1}{2}-z-k\right)}\left(N\pi^2w^2e^{-\pi i}\right)^{-k}\right.\nonumber\\
&\quad\left.+\left(N\pi^2w^2\right)^{2z}e^{N\pi^2w^2}\sum_{k=0}^\infty C_k(N\pi^2w^2)^{-k}\right\},
\end{align}
and 
\begin{align}\label{1f1nist}
{}_1F_1\left(z+\frac{1}{2};\frac{1}{2};N\pi^2w^2\right)&\sim\frac{\Gamma(1/2)}{\Gamma\left(z+\frac{1}{2}\right)}\left\{(N\pi^2w^2e^{-\pi i})^{-(z+1/2)}\sum_{k=0}^\infty\frac{(-1)^k\Gamma\left(\frac{1}{2}+z+k\right)}{k!\Gamma(k-z)}(N\pi^2w^2e^{-\pi i})^{-k}\right.\nonumber\\
&\qquad\qquad\qquad\left.+(N\pi^2w^2)^ze^{N\pi^2w^2}\sum_{k=0}^\infty C_k(N\pi^2w^2)^{-k}\right\},
\end{align}
where
\begin{align*}
C_k=-\frac{1}{k}\sum_{m=0}^{k-1}C_me_{k,m},
\end{align*}
with $C_0=1$ and
\begin{align*}
e_{k,m}=2(m-z)_{(k+1-m)}\left(z-\frac{1}{2}\right)-2z\left(\frac{1}{2}-z+m\right)_{(k+1-m)}.
\end{align*}
Upon simplifying  \eqref{before prud}, \eqref{2f2nist} and \eqref{1f1nist}, and observing that the terms containing $e^{N\pi^2w^2}$ cancel each other out, for $\frac{\pi}{4}\leq\mathrm{\arg}(w)<\frac{\pi}{2}$,
\begin{align*}
I(w,z,N)=\frac{e^{\pi i/2}}{2w^{2z+1}}\left(e^{\pi iz}-\cos(\pi z)\right)+O\left(\frac{1}{\sqrt{N}}\right)
\end{align*}
as $N\to\infty$. Employing the dominated convergence theorem to take limit $N\to\infty$ inside the double integral, we deduce that
\begin{align*}
\int_0^\infty\int_0^\infty \left(\frac{1}{(t-iw)^{2z+1}}+\frac{1}{(t+iw)^{2z+1}}\right)\cos(2\pi vt)\ dtdv&=-\frac{1}{2w^{2z+1}}\sin(\pi z).
\end{align*}
This along with \eqref{lasser region} completes the proof of the theorem for $-\frac{\pi}{2}<\arg(w)<\frac{\pi}{2}$.
\end{proof}

As discussed in the introduction, Guinand's generalization of Poisson's summation formula \cite[Theorem 1]{apg1} is critical to prove Theorem \ref{nagoyagen}. We record Guinand's result in the following theorem.
\begin{theorem}\label{Guind-pois}
\textit{If $f(x)$ is an integral, $f(x)$ tends to zero as $x\rightarrow\infty$, and $xf'(x)$ belongs to $L^p(0,\infty)$, for some p, $1<p\leq 2$, then}
\begin{align*}
\lim_{M\rightarrow\infty}\left(\sum_{m=1}^M f(m)-\int_0^M f(v)\, \mathrm{d}v\right)=\lim_{M\rightarrow\infty}\left(\sum_{m=1}^M g(m)-\int_0^M g(v)\, \mathrm{d}v\right),
\end{align*}
where
\begin{align*}
g( x )=2\int_0^{\rightarrow \infty}f(t)\cos(2\pi x t)\, \mathrm{d}t.
\end{align*}
\end{theorem}

%We now present a proof of Theorem \ref{nagoyagen}.
\begin{proof}[Theorem \textup{\ref{nagoyagen}}][]
Let 
\begin{align}
f(v)&:=\frac{1}{(v-iw)^{2z+1}}+\frac{1}{(v+iw)^{2z+1}},\quad \left(\mathrm{Re}(w)>0, \mathrm{Re}(z)>0\right),\nonumber\\
g(x)&:=2\int_0^\infty\left(\frac{1}{(v-iw)^{2z+1}}+\frac{1}{(v+iw)^{2z+1}}\right)\cos(2\pi xv)\ dv.\label{gdef}
\end{align}
Now employ Theorem \ref{Guind-pois} with $f(x)$ and $g(x)$ as above. Invoking Theorem \ref{intiszero}, we see that
\begin{align}
\sum_{n=1}^\infty g(n)&=\lim_{M\to\infty}\Bigg\{\sum_{n=1}^M\left(\frac{1}{(n-iw)^{2z+1}}+\frac{1}{(n+iw)^{2z+1}}\right)\nonumber\\
&\quad-\int_0^M\left(\frac{1}{(t-iw)^{2z+1}}+\frac{1}{(t+iw)^{2z+1}}\right)\ dt\Bigg\}-\frac{1}{w^{2z+1}}\sin(\pi z).
\end{align}
Note that series and integral on the right-hand side of the above equation exist individually in the limit $M\to\infty$. Therefore,
\begin{align}\label{el}
\sum_{n=1}^\infty g(n)=\sum_{n=1}^\infty\left(\frac{1}{(n-iw)^{2z+1}}+\frac{1}{(n+iw)^{2z+1}}\right)-\int_0^\infty\left(\frac{1}{(t-iw)^{2z+1}}+\frac{1}{(t+iw)^{2z+1}}\right)\ dt-\frac{1}{w^{2z+1}}\sin(\pi z).
\end{align}
It is easy to see that for Re$(z)>0$,
\begin{align}\label{hurwitz}
\sum_{n=1}^\infty\frac{1}{(n\mp iw)^{2z+1}}=\zeta(1+2z,1\mp iw).
%\sum_{n=1}^\infty\frac{1}{(n+iw)^{2z+1}}&=\zeta(1+2z,1+iw),
\end{align}
Also,
\begin{align}\label{cosinez}
\int_0^\infty\frac{dt}{(t\mp iw)^{2z+1}}=\frac{(\mp i)^{-2z}w^{-2z}}{2z}.
%\int_0^\infty\frac{1}{(t-iw)^{2z+1}}&=\frac{i^{-2z}w^{-2z}}{2z}
\end{align}
Substitute  \eqref{hurwitz} and \eqref{cosinez} in \eqref{el} to deduce that
\begin{align}\label{el1}
\sum_{n=1}^\infty g(n)&=\zeta(1+2z,iw)+\zeta(1+2z,-iw)-\frac{\cos(\pi z)}{zw^{2z}}-\frac{1}{w^{2z+1}}\sin(\pi z)\nonumber\\
&=\zeta(1+2z,iw)+\zeta(1+2z,-iw)-\frac{\cos(\pi z)}{zw^{2z}},
\end{align}
which follows using the fact 
\begin{equation}\label{zeta shift}
\zeta(s,a+1)=\zeta(s,a)-a^{-s}.
\end{equation}
Therefore, \eqref{gdef} and \eqref{el1} yield Theorem \ref{nagoyagen}.
%\begin{align}
%&\sum_{n=1}^\infty\int_0^\infty\left(\frac{1}{(v-iw)^{2z+1}}+\frac{1}{(v+iw)^{2z+1}}\right)\cos(2\pi nv)\ dv\nonumber\\
%&=\frac{1}{2}\left\{\zeta(1+2z,iw)+\zeta(1+2z,-iw)-\frac{\cos(\pi z)}{zw^{2z}}\right\}.\nonumber
%\end{align}
%This proves Theorem \ref{nagoyagen}.
\end{proof}

%We now obtain an equivalent representation for the generalized Raabe cosine transform $R_m(1,w)$, where $m\in\mathbb{N}\cup\{0\}$. The latter appears in the transformation of $\sum_{n=1}^{\infty}\sigma_{2m}(n)e^{-ny}$ given in Theorem \ref{a=2mcase} and it is an important result to prove Theorem \ref{a=2mcase}. 
%%It is in terms of two special functions $\textup{Shi}(z)$ and $\textup{Chi}(z)$ respectively defined by \cite[p.~150, Equation (6.2.15), (6.2.16)]{nist}
%%\begin{align}\label{shichi}
%%\mathrm{Shi}(z):=&\int_0^z\frac{\sinh(t)}{t}\ dt,\nonumber\\
%%\mathrm{Chi}(z):=&\gamma+\log(z)+\int_0^z\frac{\cosh(t)-1}{t}\ dt.
%%\end{align}
%
%\begin{lemma}\label{minusonelemma}
%Let $m\in\mathbb{N}$ and $\mathrm{Re}(w)>0$. Then
%\begin{align}\label{minusone}
%&\frac{(-1)^m(2m)!}{2}\int_0^\infty\left(\frac{1}{(t-iw)^{2m+1}}+\frac{1}{(t+iw)^{2m+1}}\right)\cos(t)\ dt\nonumber\\
%&=\sinh(w)\mathrm{Shi}(w)-\cosh(w)\mathrm{Chi}(w)+\sum_{j=1}^m(2j-1)!w^{-2j}.
%\end{align}
%\end{lemma}
\begin{proof}[Corollary \textup{\ref{cornagoya}}][]
We wish to take limit $z\to0$ in \eqref{nagoyageneqn}.
% but observe that we have poles due to Hurwitz zeta functions and the term containing cosine on the right-hand side of \eqref{nagoyageneqn}. Thus, 
To that end, we use expansions of the functions involved around $z=0$.  As $s\to1$, we have \cite[p.~1038, Formula 9.533.2]{gr}
\begin{align}
\zeta(s,a)=\frac{1}{s-1}-\psi(a)+O(|s-1|).\nonumber
\end{align}
The above equation implies that, as $z\to0$,
\begin{align}\label{huw1}
\zeta(1+2z,\pm iw)&=\frac{1}{2z}-\psi(\pm iw)+O(|z|).
\end{align}
It is easy to see that
\begin{align}\label{cosexp}
\frac{\cos(\pi z)}{zw^{2z}}=\frac{1}{z}-2\log(w)+O(|z|),
\end{align}
as $z\to0$. Using \eqref{huw1} and \eqref{cosexp}, we deduce that
\begin{align}\label{beforelimit}
\lim_{z\to0}\left(\zeta(1+2z,iw)+\zeta(1+2z,-iw)-\frac{\cos(\pi z)}{zw^{2z}}\right)&=-\psi(iw)-\psi(-iw)+2\log(w).
\end{align}
Let $z\to0$ on both sides of \eqref{nagoyageneqn} and use \eqref{beforelimit} so that
\begin{align}\label{beforefunct}
2\sum_{n=1}^\infty \int_0^\infty\frac{v\cos(2\pi n v)}{v^2+w^2}\ dv=\frac{1}{2}\left\{2\log(w)-\left(\psi(iw)+\psi(-iw)\right)\right\}.
\end{align}
%The recursion relation for $\psi(a)$ is \cite[p.~54]{temme}
%\begin{align}\label{psifunct}
%\psi(a+1)=\frac{1}{a}+\psi(a).
%\end{align}
Make the change of variable $2\pi nv=t$ on the right-hand side of \eqref{beforefunct} to arrive at
\begin{align*}
2\sum_{n=1}^\infty \int_0^\infty\frac{t\cos(t)}{t^2+(2\pi w)^2n^2}\ dt=\log(w)-\frac{1}{2}\left(\psi(iw)+\psi(-iw)\right).
\end{align*}
Finally let $w=u/(2\pi)$ in the above equation to conclude the proof of the corollary.
\end{proof}

Theorem \ref{minusonelemma} is proved next.
\begin{proof}[Theorem \textup{\ref{minusonelemma}}][]
From \cite[Lemma 9.1]{dkk1}, for Re$(w)>0$, we have
\begin{equation}\label{zero}
\int_{0}^{\infty}\frac{t\cos t\, dt}{t^2+w^2}=\sinh(w)\mathrm{Shi}(w)-\cosh(w)\mathrm{Chi}(w).
\end{equation}
Now \eqref{minusone} follows by expanding $\frac{t}{(t^2+w^2)}$ in partial fractions, that is, by writing $\frac{t}{t^2+w^2}=\frac{1}{2}\left(\frac{1}{t-iw}+\frac{1}{t+iw}\right)$, and then by performing integration by parts $2m$ times the left-hand side of \eqref{zero}. 
\end{proof}

\section{Proof of our generalization of a formula of Ramanujan}\label{section ram thm}
We begin with the following result of Hardy \cite[pp.~56-57]{hardypaper2}. This result helps us justify the interchange of the order of the summation and integration having principal values, and will be employed in the proof of Theorem \ref{ram with a}.
\begin{proposition}\label{hardy theorem}
Let 
\begin{align*}
S(x)=\sum_{k=0}^\infty u_k(x)
\end{align*}
be a series whose terms are functions of $x$ and is convergent with the possible exception of a closed enumerable set of points for values of $x$ in a finite interval $(a,A)$. Let $\alpha$ denote one such point in this set. If 
\begin{enumerate}
\item the series $S(x)$ is integrable term by term over any part of $(a,A)$ which does not include $\alpha$,  \label{cond 1}
\item the function 
\begin{align}
F(x)=\sum_{k=0}^\infty\mathrm{PV}\int_a^x u_k(t) dt\nonumber
\end{align}
is a continuous function of $x$ except at $\alpha$, \label{cond 2} and 
\item 
\begin{align}\label{limit zero hypothesis}
\lim_{\epsilon\to0}\left\{F(\alpha-\epsilon)-F(\alpha+\epsilon)\right\}=0.
\end{align}
\end{enumerate}
Then, one can interchange the order of summation and integration, namely,
\begin{align}
\mathrm{PV}\int_a^A\sum_{k=0}^\infty u_k(t) dt=\sum_{k=0}^\infty\mathrm{PV}\int_a^Au_k(t).\nonumber
\end{align}
\end{proposition}

\begin{remark}\label{remark2}
We note that \cite[pp.~58--59, Section 7]{hardypaper2} \textup{(}also see \cite[p.~27]{hardy}\textup{)} if
\begin{align}\label{un}
u_k(x)=\frac{v_k(x)}{x-\alpha},
\end{align}
where $v_k(x)$ is a function of $x$ and has a continuous derivative for all $x\in[a,A]$, then 
\begin{align}
\mathrm{PV}\int_{\alpha-\epsilon}^{\alpha+\epsilon}u_k(x)dx=2\epsilon v_k'(\alpha+\mu),\ \mathrm{for\ some}\ \mu\in[-\epsilon,\epsilon].\nonumber
\end{align}
Also if $|v_k'(x)|<V_k$ for all values $x\in[a,A]$, $V_k$ being independent of $x$ and $\sum_{k=0}^\infty V_k$ is convergent, then the condition \eqref{limit zero hypothesis} holds true for $u_k(x)$ given in \eqref{un}.
\end{remark}

In the next lemma, we justify the interchange of the order of the summation and principal value integral.
\begin{lemma}\label{intechange of PV and summation}
Let $k\in\mathbb{N}$, $0<a\leq1$ and $\mathrm{Re}(y)>0$. For $\mathrm{Re}(s)>2$, we have
%\footnote{Note that we did not write the closed form evaluation of the series on the right-hand side of \eqref{interchange lemma}, given in \eqref{sin k}, because in the proof of Theorem \ref{ram with a} we will be needing this result in its current form.}
\begin{align}\label{interchange lemma}
&\sum_{k=1}^\infty\sin\left(2\pi ak\right) \mathrm{PV}\int_0^\infty x^{s-1}e^{-4\pi^2 kx/y}\cot(\pi x)dx= \mathrm{PV}\int_0^\infty\left( \sum_{k=1}^\infty\sin\left(2\pi ak\right)e^{-4\pi^2 kx/y}\right)x^{s-1}\cot(\pi x)dx.
%&=\frac{1}{2i}\mathrm{PV}\int_0^\infty \left(\frac{1}{e^{\frac{4\pi^2 x}{y}-2\pi ia}-1}-\frac{1}{e^{\frac{4\pi^2 x}{y}+2\pi ia}-1}\right)\cot(\pi x) x^{s-1}dx.
\end{align}
\end{lemma}
\begin{proof}
	Note that the presence of $\cot(\pi x)$ implies infinitely many singularities of the integrand on the left side of \eqref{interchange lemma}. To handle this integral efficiently, we use 
\begin{align}\label{cotx}
\pi\cot(\pi x)=\frac{1}{x}+\sum_{n=1}^\infty\frac{2x}{x^2-n^2},
\end{align}	
so that 
\begin{align}\label{first interchange}
\sum_{k=1}^\infty\sin\left(2\pi ak\right) \mathrm{PV}\int_0^\infty x^{s-1}e^{-4\pi^2 kx/y}\cot(\pi x)dx&=\frac{1}{\pi}\sum_{k=1}^\infty\sin\left(2\pi ak\right) \int_0^\infty x^{s-2}e^{-4\pi^2 kx/y}dx\nonumber\\&\quad+\frac{2}{\pi}\sum_{k=1}^\infty\sin\left(2\pi ak\right) \mathrm{PV}\int_0^\infty x^{s}e^{-4\pi^2 kx/y}\sum_{n=1}^\infty\frac{1}{x^2-n^2} dx.
\end{align}
We can interchange the order of summation and integration in the first expression on the right-hand side of \eqref{first interchange} by easily employing \cite[p.~30, Theorem 2.1]{temme}.  The delicate part is to show the same for the second expression on the right,  which is done next. We first show that
\begin{align}\label{inter of n and int}
\mathrm{PV}\int_0^\infty x^{s}e^{-4\pi^2 kx/y}\sum_{n=1}^\infty\frac{1}{x^2-n^2} dx=\sum_{n=1}^\infty\mathrm{PV}\int_0^\infty \frac{x^{s}e^{-4\pi^2 kx/y}}{x^2-n^2} dx.
\end{align}
The ingenious argument given in \cite[pp.~909-911]{bdrz} can be adapted here as well to prove  the above claim. We give the complete details though to make the paper self-contained.

Let $w(t)\in C_0^\infty$ be a smooth function such that $0\leq w(t)\leq1$, $\forall\ t\in\mathbb{R}$, $w(t)$ has compact support in $\left(-\frac{1}{3},\frac{1}{3}\right)$, and $w(t)=1,\ t\in\left(-\frac{1}{4},\frac{1}{4}\right)$. Note that the right-hand side of \eqref{inter of n and int} can be rewritten as
\begin{align}\label{plus two}
\sum_{n=1}^\infty\mathrm{PV}\int_0^\infty \frac{x^{s}e^{-4\pi^2 kx/y}}{x^2-n^2} dx&=\sum_{n=1}^\infty\int_0^\infty x^{s}e^{-4\pi^2 kx/y}\frac{(1-w(x-n))}{x^2-n^2} dx+\sum_{n=1}^\infty\mathrm{PV}\int_0^\infty x^{s}e^{-4\pi^2 kx/y}\frac{w(x-n)}{x^2-n^2} dx.
\end{align}
Again, an easy application of \cite[p.~30, Theorem 2.1]{temme} allows us to interchange the order of summation and integration in the first expression of \eqref{plus two}. If $m$ is a positive integer such that $m-\frac{1}{2}\leq x\leq m+\frac{1}{2}$, then 
\begin{align}\label{in terms of m}
\sum_{n=1}^\infty \frac{w(x-n)}{x^2-n^2}=\frac{w(x-m)}{x^2-m^2}.
\end{align}
Hence, using \eqref{in terms of m} in the second step below,  we have
\begin{align}
\mathrm{PV}\int_0^\infty x^{s}e^{-4\pi^2 kx/y}\sum_{n=1}^\infty\frac{w(x-n)}{x^2-n^2} dx&=\sum_{m=1}^\infty \mathrm{PV}\int_{m-1/2}^{m+1/2}x^{s}e^{-4\pi^2 kx/y}\sum_{n=1}^\infty\frac{w(x-n)}{x^2-n^2} dx\nonumber\\
&=\sum_{m=1}^\infty \mathrm{PV}\int_{m-1/2}^{m+1/2}x^{s}e^{-4\pi^2 kx/y}\frac{w(x-m)}{x^2-m^2} dx\nonumber\\
&=\sum_{m=1}^\infty \mathrm{PV}\int_{0}^{\infty}x^{s}e^{-4\pi^2 kx/y}\frac{w(x-m)}{x^2-m^2} dx.\nonumber
\end{align}
The above fact along with \eqref{plus two} gives
\begin{align*}
\sum_{n=1}^\infty\mathrm{PV}\int_0^\infty \frac{x^{s}e^{-4\pi^2 kx/y}}{x^2-n^2} dx&=\mathrm{PV}\int_0^\infty x^{s}e^{-4\pi^2 kx/y}\sum_{n=1}^\infty\frac{(1-w(x-n))}{x^2-n^2} dx+\mathrm{PV}\int_0^\infty x^{s}e^{-4\pi^2 kx/y}\sum_{n=1}^\infty\frac{w(x-n)}{x^2-n^2} dx\\
&=\mathrm{PV}\int_0^\infty x^{s}e^{-4\pi^2 kx/y}\sum_{n=1}^\infty\frac{1}{x^2-n^2} dx.
\end{align*}
This proves the claim in \eqref{inter of n and int}. Therefore, we can write
\begin{align}
\sum_{k=1}^\infty\sin\left(2\pi ak\right) \mathrm{PV}\int_0^\infty x^{s}e^{-4\pi^2 kx/y}\sum_{n=1}^\infty\frac{1}{x^2-n^2} dx&=\sum_{k=1}^\infty\sin\left(2\pi ak\right) \sum_{n=1}^\infty\mathrm{PV}\int_0^\infty \frac{x^{s}e^{-4\pi^2 kx/y}}{x^2-n^2} dx.\nonumber
\end{align} 
Fubini's theorem allows us to interchange the order of the double sum on the right-hand side of the above expression so as to obtain
\begin{align}\label{over n k}
&\sum_{k=1}^\infty\sin\left(2\pi ak\right) \mathrm{PV}\int_0^\infty x^{s}e^{-4\pi^2 kx/y}\sum_{n=1}^\infty\frac{1}{x^2-n^2} dx=\sum_{n=1}^\infty\sum_{k=1}^\infty\sin\left(2\pi ak\right) \mathrm{PV}\int_0^\infty \frac{x^{s}e^{-4\pi^2 kx/y}}{x^2-n^2} dx.
\end{align}
Now
\begin{align}\label{split}
\sum_{k=1}^\infty\sin\left(2\pi ak\right) \mathrm{PV}\int_0^\infty \frac{x^{s}e^{-4\pi^2 kx/y}}{x^2-n^2} dx&=\sum_{k=1}^\infty\sin\left(2\pi ak\right) \left\{\left(\int_{0}^\delta+\int_{n+1}^\infty\right)\frac{x^{s}e^{-4\pi^2 kx/y}}{x^2-n^2} dx\right.\nonumber\\
&\quad\left.+\frac{1}{2}\int_\delta^{n+1} \frac{x^{s-1}e^{-4\pi^2 kx/y}}{x+n} dx+\frac{1}{2}\mathrm{PV}\int_\delta^{n+1} \frac{x^{s-1}e^{-4\pi^2 kx/y}}{x-n} dx\right\},
\end{align}
where $0<\delta<1$. Note that there is no need to take principal value for the first three integrals on the right-hand side of \eqref{split}. Therefore, it is easy to take the summation inside these integrals using the standard techniques, for example, \cite[p.~30, theorem 2.1]{temme}.  To interchange the order of summation and the last integral in \eqref{split}, we now show that the hypotheses of Proposition \ref{hardy theorem} are satisfied. Let us define
\begin{align}\label{u and v}
u_k(x):=\frac{v_k(x)}{x-n}\quad \mathrm{and}\quad v_k(x):=x^{s-1}e^{-4\pi^2 kx/y}\sin\left(2\pi ak\right).
\end{align}
It is easy to see that the conditions \eqref{cond 1} and \eqref{cond 2}  of Proposition \ref{hardy theorem} are satisfied with $u_k(x)$ being defined in \eqref{u and v}. To fulfill \eqref{limit zero hypothesis}, we show that the equivalent condition discussed in Remark \ref{remark2} is satisfied. To that end, observe that $x\in[\delta,n+1]$ and use $e^{-x}>3!/x^3, x>0$, so that
\begin{align}
|v_k'(x)|&<\left|x^{s-2}e^{-4\pi^2 kx/y}\left(s-1-\frac{4\pi^2kx}{y}\right)\right|<\frac{x^{\mathrm{Re}(s)-5}}{(4\pi^2/y)^3}\frac{3!}{k^3}\left(|s-1|+\frac{4\pi^2kx}{y}\right)\nonumber\\
&<\frac{M}{(4\pi^2/y)^3}\frac{3!}{k^3}\left(|s-1|+\frac{4\pi^2k(n+1)}{y}\right)\nonumber\\
&=:V_k,\nonumber
\end{align}
where we used the fact that the function $x^{\mathrm{Re}(s)-5}$ is continuous on the compact interval $[\delta,n+1]$, and hence bounded  by some constant $M>0$ (which may depend on $\delta$ and $n$). Since the series $\sum_{k=1}^\infty V_k$ converges, all conditions of Proposition \ref{hardy theorem} are satisfied. Hence we can interchange the order of summation and integration even in the case of the last integral of \eqref{split}. This fact along with the discussion following \eqref{split} implies that
\begin{align}\label{interchanged over k}
\sum_{k=1}^\infty\sin\left(2\pi ak\right) \mathrm{PV}\int_0^\infty \frac{x^{s}e^{-4\pi^2 kx/y}}{x^2-n^2} dx&=\mathrm{PV}\int_0^\infty \left(\sum_{k=1}^\infty\sin\left(2\pi ak\right)e^{-4\pi^2 kx/y}\right)\frac{x^{s}}{x^2-n^2} dx.
\end{align}
Using the fact $\sin(\theta)=(e^{i\theta}-e^{-i\theta})/(2i)$, we find 
\begin{align}\label{sin k}
\sum_{k=1}^\infty\sin\left(2\pi ak\right)e^{-4\pi^2 xk/y}&=\frac{1}{2i}\sum_{k=1}^\infty e^{-\left(\frac{4\pi^2 x}{y}-2\pi ia\right)k}-\frac{1}{2i}\sum_{k=1}^\infty e^{-\left(\frac{4\pi^2 x}{y}+2\pi ia\right)k}\nonumber\\
&=\frac{1}{2i}\left(\frac{1}{e^{\frac{4\pi^2 x}{y}-2\pi ia}-1}-\frac{1}{e^{\frac{4\pi^2 x}{y}+2\pi ia}-1}\right).
\end{align}
Substitute the above value in \eqref{interchanged over k} to arrive at
\begin{align}\label{k simp}
\sum_{k=1}^\infty\sin\left(2\pi ak\right) \mathrm{PV}\int_0^\infty \frac{x^{s}e^{-4\pi^2 kx/y}}{x^2-n^2} dx&=\frac{1}{2i}\mathrm{PV}\int_0^\infty \left(\frac{1}{e^{\frac{4\pi^2 x}{y}-2\pi ia}-1}-\frac{1}{e^{\frac{4\pi^2 x}{y}+2\pi ia}-1}\right)\frac{x^{s}}{x^2-n^2} dx.
\end{align}
Equations \eqref{over n k} and \eqref{k simp} yield
\begin{align*}
&\sum_{k=1}^\infty\sin\left(2\pi ak\right) \mathrm{PV}\int_0^\infty x^{s}e^{-4\pi^2 kx/y}\sum_{n=1}^\infty\frac{1}{x^2-n^2} dx\nonumber\\
&=\frac{1}{2i}\sum_{n=1}^\infty\mathrm{PV}\int_0^\infty \left(\frac{1}{e^{\frac{4\pi^2 x}{y}-2\pi ia}-1}-\frac{1}{e^{\frac{4\pi^2 x}{y}+2\pi ia}-1}\right)\frac{x^{s}}{x^2-n^2} dx.
\end{align*}
Again employing the trick that we used after \eqref{inter of n and int} to interchange the order of the summation and integration, one can take the sum over $n$ inside the integral on the right-hand side of the above equation to deduce that
\begin{align}\label{second interchange}
&\sum_{k=1}^\infty\sin\left(2\pi ak\right) \mathrm{PV}\int_0^\infty x^{s}e^{-4\pi^2 kx/y}\sum_{n=1}^\infty\frac{1}{x^2-n^2} dx\nonumber\\
&=\frac{1}{2i}\mathrm{PV}\int_0^\infty \left(\frac{1}{e^{\frac{4\pi^2 x}{y}-2\pi ia}-1}-\frac{1}{e^{\frac{4\pi^2 x}{y}+2\pi ia}-1}\right)\sum_{n=1}^\infty\frac{1}{x^2-n^2} x^sdx.
\end{align}
Substituting \eqref{second interchange} in \eqref{first interchange}, we obtain
\begin{align}
&\sum_{k=1}^\infty\sin\left(2\pi ak\right) \mathrm{PV}\int_0^\infty x^{s-1}e^{-4\pi^2 kx/y}\cot(\pi x)dx\nonumber\\
&=\frac{1}{\pi}\int_0^\infty x^{s-2}\sum_{k=1}^\infty\sin\left(2\pi ak\right)e^{-4\pi^2 kx/y}dx+\frac{1}{i\pi}\mathrm{PV}\int_0^\infty \left(\frac{1}{e^{\frac{4\pi^2 x}{y}-2\pi ia}-1}-\frac{1}{e^{\frac{4\pi^2 x}{y}+2\pi ia}-1}\right)\sum_{n=1}^\infty\frac{1}{x^2-n^2} x^sdx\nonumber\\
&=\frac{1}{\pi}\int_0^\infty \left(\sum_{k=1}^\infty\sin\left(2\pi ak\right)e^{-4\pi^2 kx/y}\right)x^{s-2}dx+\frac{2}{\pi}\int_0^\infty \left(\sum_{k=1}^\infty\sin\left(2\pi ak\right)e^{-4\pi^2 kx/y}\right)\sum_{n=1}^\infty\frac{1}{x^2-n^2} x^sdx,\nonumber
\end{align}
where in the ultimate step we again used \eqref{sin k}. Finally employing \eqref{cotx} in the above equation, we arrive at \eqref{interchange lemma}.
\end{proof}

We have now collected all ingredients to give a proof of our generalization of Ramanujan's formula.
%The proof of Theorem \ref{ramanujan thm} is presented next.
%
%\begin{proof}[Theorem \text{\ref{ramanujan thm}}][]
\begin{proof}[Theorem \textup{\ref{ram with a}}][]
%We first explain the reason for having the condition  Re$(s)>2$ for $0\leq a<1$, and Re$(s)>1$ for $0< a<1$. Note that all the series occurring in the formula are convergent for Re$(s)>1$. However, the principal value integral is convergent for Re$(s)>2$ when $a=0$, and for Re$(s)>1$ for $0<a<1$ as can be seen from the fact that as $x\to0$,
%\begin{align}\label{bound}
%\frac{1}{e^{2\pi x\pm2\pi ia}-1}=
%\begin{cases}
%O\left(\frac{1}{x}\right), &\mathrm{if}\ a=0,\\
%O\left(\frac{1}{e^{\pm2\pi ia}-1}\right), & \mathrm{if}\ 0<a<1.
%\end{cases}
%\end{align}

Letting $\tau= iyj/(2\pi)$, Re$(y)>0$, in \eqref{lipscheqn}, then taking summation over $j\geq1$, and then employing the series definition of the Hurwitz zeta function for Re$(s)>1$,  we obtain\footnote{The case $a=0$ of \eqref{4.3} reduces to a result of Kuylenstierna \cite[Equation (7)]{ku}.}
\begin{align}\label{4.3}
\sum_{n=1}^\infty\frac{(n-a)^{s-1}}{e^{(n-a)y}-1}&=\frac{\Gamma(s)}{(-2\pi i)^s}\sum_{k\in\mathbb{Z}}e^{2\pi iak}\sum_{j=1}^\infty\frac{1}{\left(k+\frac{ijy}{2\pi}\right)^s}\nonumber\\
&=\frac{\Gamma(s)}{y^s}\sum_{k\in\mathbb{Z}}e^{2\pi iak}\zeta\left(s,1-\frac{2\pi ik}{y}\right)\nonumber\\
&=\frac{\Gamma(s)\zeta(s)}{y^s}+\frac{\Gamma(s)}{y^s}\sum_{k=1}^\infty\left\{e^{2\pi iak}\zeta\left(s,1-\frac{2\pi i k}{y}\right)+e^{-2\pi iak}\zeta\left(s,1+\frac{2\pi i k}{y}\right)\right\}.
\end{align}
Invoking the well-known formula \cite[p.~609, Formula 25.11.25]{nist}
\begin{align}\label{mt of hz}
\Gamma(z)\zeta(z,a)=\int_0^\infty\frac{e^{-ax}}{1-e^{-x}}x^{z-1}dx,\qquad (\mathrm{Re}(z)>1,\ \mathrm{Re}(a)>0)
\end{align}
in \eqref{4.3}, we obtain
\begin{align}\label{integral}
\sum_{n=1}^\infty\frac{(n-a)^{s-1}}{e^{(n-a)y}-1}&=\frac{\Gamma(s)\zeta(s)}{y^s}+\frac{1}{y^s}\sum_{k=1}^\infty\int_0^\infty\left(e^{i\left(2\pi ak+\frac{2\pi kt}{y}\right)}+e^{-i\left(2\pi ak+\frac{2\pi kt}{y}\right)}\right)\frac{t^{s-1}}{e^{t}-1}dt\nonumber\\
&=\frac{\Gamma(s)\zeta(s)}{y^s}+\frac{2}{y^s}\sum_{k=1}^\infty\int_0^\infty\cos\left(2\pi a k+\frac{2\pi kt}{y}\right)\frac{t^{s-1}}{e^{t}-1}dt\nonumber\\
&=\frac{\Gamma(s)\zeta(s)}{y^s}+\frac{2}{y^s}\sum_{k=1}^\infty\cos(2\pi a k)\int_0^\infty\cos\left(\frac{2\pi kt}{y}\right)\frac{t^{s-1}}{e^{t}-1}dt\nonumber\\
&\hspace{1.8cm}-\frac{2}{y^s}\sum_{k=1}^\infty\sin(2\pi a k)\int_0^\infty\sin\left(\frac{2\pi kt}{y}\right)\frac{t^{s-1}}{e^{t}-1}dt.
\end{align}
Our next goal is to evaluate the integrals in \eqref{integral}. From \cite[p.~42, Formula 1.5.2]{ober}, for $0<\mathrm{Re}(z)<1$, we have
\begin{align}
\int_0^\infty\cos(x)x^{z-1}dx=\Gamma(z)\cos\left(\frac{\pi z}{2}\right).\nonumber
\end{align}
Making the change of variable $x=\frac{2\pi k t}{y}$ and replacing $z$ by $s-1+z$ in the above result, we get, for $1-\mathrm{Re}(s)<\mathrm{Re}(z)<2-\mathrm{Re}(s)$,
\begin{align}\label{mt of cosine fn}
\int_0^\infty\cos\left(\frac{2\pi kt}{y}\right)t^{s-1}t^{z-1}dt=\left(\frac{2\pi k}{y}\right)^{1-s-z}\Gamma(s-1+z)\sin\left(\frac{\pi}{2}(s+z)\right).
\end{align}
Equation \eqref{mt of hz} with $a=1$, \eqref{mt of cosine fn}, and an application of Parseval's formula \cite[p.~83, Equation (3.1.14)]{parsvel} gives, for $1-\mathrm{Re}(s)<c=\mathrm{Re}(z)<\textup{min}\left(0, 2-\mathrm{Re}(s)\right)$, 
\begin{align}\label{i1+i2}
\int_0^\infty\cos\left(\frac{2\pi kt}{y}\right)\frac{t^{s-1}}{e^t-1}dt
&=\left(\frac{2\pi k}{y}\right)^{1-s}\frac{1}{2\pi i}\int_{(c)}\Gamma(s-1+z)\sin\left(\frac{\pi}{2}(s+z)\right)\Gamma(1-z)\zeta(1-z)\left(\frac{2\pi k}{y}\right)^{-z}dz\nonumber\\
%By using the simple fact $\sin(A+B)=\cos(A)\sin(B)+\sin(A)\cos(B)$, the above integral can be written as
&=\left(\frac{2\pi k}{y}\right)^{1-s}\left\{\cos\left(\frac{\pi s}{2}\right)I_1(y,s)+\sin\left(\frac{\pi s}{2}\right)I_2(y,s)\right\},
\end{align}
where 
\begin{align}
I_1(y,s)&:=\frac{1}{2\pi i}\int_{(c)}\Gamma(s-1+z)\sin\left(\frac{\pi z}{2}\right)\Gamma(1-z)\zeta(1-z)\left(\frac{2\pi k}{y}\right)^{-z}dz,\label{i1ys}\\
I_2(y,s)&:=\frac{1}{2\pi i}\int_{(c)}\Gamma(s-1+z)\cos\left(\frac{\pi z}{2}\right)\Gamma(1-z)\zeta(1-z)\left(\frac{2\pi k}{y}\right)^{-z}dz.\label{i2ys}
\end{align}
Similarly, using the formula \cite[p.~42, Formula 1.5.1]{ober}
\begin{align*}
\int_0^\infty\sin(x)x^{z-1}dx=\Gamma(z)\sin\left(\frac{\pi z}{2}\right),\quad (-1<\mathrm{Re}(z)<1),
\end{align*}
it can be seen that for $-\mathrm{Re}(s)<c=\mathrm{Re}(z)<\textup{min}\left(0, 2-\mathrm{Re}(s)\right)$, 
\begin{align}\label{sin i1+i2}
\int_0^\infty\sin\left(\frac{2\pi kt}{y}\right)\frac{t^{s-1}}{e^t-1}dt
=\left(\frac{2\pi k}{y}\right)^{1-s}\left\{\sin\left(\frac{\pi s}{2}\right)I_1(y,s)-\cos\left(\frac{\pi s}{2}\right)I_2(y,s)\right\}.
\end{align}
We first evaluate $I_1(y,s)$. Apply the functional equation of the Riemann zeta function \cite[p.~603, Formula 25.4.2]{nist}
\begin{align}\label{zeta fe}
\zeta(s)=2^{s}\pi^{s-1}\Gamma(1-s)\zeta(1-s)\sin\left(\frac{\pi s}{2}\right),
\end{align}
in \eqref{i1ys} to see that
\begin{align}
I_1(y,s)=\frac{\pi}{2\pi i}\int_{(c)}\Gamma(s-1+z)\zeta(z)\left(\frac{4\pi^2 k}{y}\right)^{-z}dz.\nonumber
\end{align}
We want to use the series definition of $\zeta(z)$ to further simplify the above integral. Therefore we shift the line of integration to $d=\mathrm{Re}(z)>1$ and use residue theorem thereby obtaining
\begin{align}\label{apply residue}
I_1(y,s)&=\frac{\pi}{2\pi i}\int_{(d)}\Gamma(s-1+z)\zeta(z)\left(\frac{4\pi^2 k}{y}\right)^{-z}dz-\frac{y\Gamma(s)}{4\pi k}\nonumber\\
&=\pi\sum_{n=1}^\infty\frac{1}{2\pi i}\int_{(d)}\Gamma(s-1+z)\left(\frac{4\pi^2 nk}{y}\right)^{-z}dz-\frac{y\Gamma(s)}{4\pi k}\nonumber\\
&=\pi\left(\frac{4\pi^2 k}{y}\right)^{s-1}\sum_{n=1}^\infty n^{s-1}e^{-\frac{4\pi^2 nk}{y}}-\frac{y\Gamma(s)}{4\pi k},
\end{align}
where in the last step, we used
\begin{align}\label{invgamma}
	e^{-x}=\frac{1}{2\pi i}\int_{(\lambda)}\Gamma(z) x^{-z}dz\hspace{8mm}(\lambda>0).
\end{align}
We now focus on representing the other integral $I_2(y,s)$ in terms of an equivalent integral; see \eqref{i2 evaluated} below.   Again an application of \eqref{zeta fe} in \eqref{i2ys} yields
\begin{align}\label{i2 almost1}
I_2(y,s)=\frac{\pi}{2\pi i}\int_{(c)}\Gamma(s-1+z)\zeta(z)\cot\left(\frac{\pi z}{2}\right)\left(\frac{4\pi^2 k}{y}\right)^{-z}dz.
\end{align}
 If we shift the line of integration from Re$(z)=c$, where $1-\mathrm{Re}(s)<c<\min(0,2-\mathrm{Re}(s))$,   to $1<\mathrm{Re}(z)=d<2$, we encounter a simple pole at $z=0$ of the integrand in the integral of \eqref{i2 almost1} due to $\cot(\pi z/2)$. (Note that the pole of $\zeta(z)$ at $z=1$ is annihilated by the zero of $\cot(\pi z/2)$ at $z=1$.) Note that the integrals along the horizontal segments vanish using Stirling's formula in the vertical strip $p\leq\sigma\leq q$ \cite[p.~224]{cop}:
\begin{equation}\label{strivert}
  |\Gamma(s)|=\sqrt{2\pi}|t|^{\sigma-\frac{1}{2}}e^{-\frac{1}{2}\pi |t|}\left(1+O\left(\frac{1}{|t|}\right)\right)
\end{equation}
as $|t|\to \infty$. Therefore, by the residue theorem and \eqref{i2 almost1}, we have
 \begin{align}\label{i2 almost}
I_2(y,s)=\Gamma(s-1)+\frac{\pi}{2\pi i}\int_{(d)}\Gamma(s-1+z)\zeta(z)\cot\left(\frac{\pi z}{2}\right)\left(\frac{4\pi^2 k}{y}\right)^{-z}dz.
\end{align} 
Note that we can use the series definition of $\zeta(s)$ in \eqref{i2 almost} and then interchange the order of the summation and integration so as to obtain
\begin{align}\label{i2 almost2}
I_2(y,s)&=\Gamma(s-1)+\pi\sum_{n=1}^\infty\frac{1}{2\pi i}\int_{(d)}\Gamma(s-1+z)\cot\left(\frac{\pi z}{2}\right)\left(\frac{4\pi^2 nk}{y}\right)^{-z}dz.
\end{align}
From \cite[p.~182, Formula 2.4.4]{ober}, for $-1<c_2<1$, we have
\begin{align*}
\frac{1}{2\pi i}\int_{(c_2)}\tan\left(\frac{\pi z}{2}\right)x^{-z}dz=\frac{2}{\pi}\frac{x}{x^2-1},\quad (x\neq\pm1).
\end{align*}
Replacing $x$ by $x/n$ in the above result gives
\begin{align}\label{c2}
\frac{1}{2\pi i}\int_{(c_2)}\tan\left(\frac{\pi z}{2}\right)n^{z-1}x^{-z}dz=\frac{2}{\pi}\frac{x}{x^2-n^2},\quad (x\neq\pm n).
\end{align}
For $c_1>1-\mathrm{Re}(s)$, equation \eqref{invgamma} implies that
\begin{align}\label{c1}
\frac{1}{2\pi i}\int_{(c_1)}\Gamma(s-1+z)\left(\frac{4\pi^2k}{y}\right)^{-z}x^{-z}dz=e^{-\frac{4\pi^2 xk}{y}}\left(\frac{4\pi^2 xk}{y}\right)^{s-1}.
\end{align}
We next want to invoke Parseval's formula \cite[p.~83, Equation (3.1.11)]{parsvel} for the functions in \eqref{c2} and \eqref{c1}. For that we need to justify the following\footnote{Note that one has to justify the interchange of the order of the integrals in third step of \cite[p.~83]{parsvel} to use  Parseval's formula. The conditions under which it can be done are given after \cite[p.~83, Equation (3.1.11)]{parsvel}. But one of our integrals is a principal value integral, therefore, we need to justify this interchange of the order of the integration.}
\begin{align}\label{int of int}
\frac{1}{2\pi i}\int_{(d)}\mathrm{PV}\int_0^\infty\Gamma(s-1+z)\frac{x^{z-1}}{1-n^2x^2}dxdz=\mathrm{PV}\int_0^\infty\frac{1}{2\pi i}\int_{(d)}\Gamma(s-1+z)\frac{x^{z-1}}{1-n^2x^2}dzdx,
\end{align}
where $d>0$. Making the change of variable $z=d+it$, we see that
\begin{align}\label{intrechange1}
\frac{1}{2\pi i}\int_{(d)}\mathrm{PV}\int_0^\infty\Gamma(s-1+z)\frac{x^{z-1}}{1-n^2x^2}dxdz&=\frac{1}{2\pi}\int_{-\infty}^\infty\mathrm{PV}\int_0^\infty\Gamma(s-1+d+it)\frac{x^{d+it-1}}{1-n^2x^2}dxdt\nonumber\\
&=\frac{1}{2\pi}\int_{-\infty}^\infty\mathrm{PV}\int_0^{\frac{1}{n}+\epsilon}\Gamma(s-1+d+it)\frac{x^{d+it-1}}{1-n^2x^2}dxdt\nonumber\\
&+\frac{1}{2\pi}\int_{-\infty}^\infty\int_{\frac{1}{n}+\epsilon}^\infty\Gamma(s-1+d+it)\frac{x^{d+it-1}}{1-n^2x^2}dxdt.
\end{align}
Note that the inner integral in the second expression on the right-hand side is a usual improper integral. Therefore, we can interchange the order of the integration by standard methods \cite[p.~30, Theorem 2.2]{temme}. To justify the same in the first double integral on the right, we proceed as follows. Observe that 
%\begin{align}
%\int_{0}^\infty\mathrm{PV}\int_0^{\frac{1}{n}+\epsilon}\Gamma(s-1+d+it)\frac{x^{d+it-1}}{1-n^2x^2}dxdt&=\sum_{m=0}^\infty\int_{m}^{m+1}\mathrm{PV}\int_0^{\frac{1}{n}+\epsilon}\Gamma(s-1+d+it)\frac{x^{d+it-1}}{1-n^2x^2}dxdt.\nonumber
%\end{align}
\begin{align}\label{justify}
\frac{1}{2\pi}\int_{-\infty}^\infty\mathrm{PV}\int_0^{\frac{1}{n}+\epsilon}\Gamma(s-1+d+it)\frac{x^{d+it-1}}{1-n^2x^2}dxdt&=\frac{1}{2\pi}\int_{0}^\infty\mathrm{PV}\int_0^{\frac{1}{n}+\epsilon}\Gamma(s-1+d+it)\frac{x^{d+it-1}}{1-n^2x^2}dxdt\nonumber\\
&\quad+\frac{1}{2\pi}\int_{0}^\infty\mathrm{PV}\int_0^{\frac{1}{n}+\epsilon}\Gamma(s-1+d-it)\frac{x^{d-it-1}}{1-n^2x^2}dxdt.
\end{align}
We justify the interchange of the order of integration only for the first double integral. That for the second one can be similarly justified. To that end, for $B>0$, 
\begin{align}
\int_{0}^B\mathrm{PV}\int_0^{\frac{1}{n}+\epsilon}\Gamma(s-1+d+it)\frac{x^{d+it-1}}{1-n^2x^2}dxdt&=\mathrm{PV}\int_0^{\frac{1}{n}+\epsilon}\int_0^B\Gamma(s-1+d+it)\frac{x^{d+it-1}}{1-n^2x^2}dxdt,
\end{align}
using Hardy's result \cite[p.~94, Theorem 6]{hardypaper3}. 
%hence we can interchange the order of the integration as
%\begin{align}\label{intrechange2}
%\int_{0}^\infty\mathrm{PV}\int_0^{\frac{1}{n}+\epsilon}\Gamma(s-1+d+it)\frac{x^{d+it-1}}{1-n^2x^2}dxdt&=\sum_{m=0}^\infty\mathrm{PV}\int_0^{\frac{1}{n}+\epsilon}\int_{m}^{m+1}\Gamma(s-1+d+it)\frac{x^{d+it-1}}{1-n^2x^2}dtdx\nonumber\\
%&=\mathrm{PV}\int_0^{\frac{1}{n}+\epsilon}\sum_{m=0}^\infty\int_{m}^{m+1}\Gamma(s-1+d+it)\frac{x^{d+it-1}}{1-n^2x^2}dtdx\nonumber\\
%&=\mathrm{PV}\int_0^{\frac{1}{n}+\epsilon}\int_{0}^{\infty}\Gamma(s-1+d+it)\frac{x^{d+it-1}}{1-n^2x^2}dtdx,
%\end{align}
%\begin{align}\label{intrechange2}
%\int_{0}^B\mathrm{PV}\int_0^{\frac{1}{n}+\epsilon}\Gamma(s-1+d+it)\frac{x^{d+it-1}}{1-n^2x^2}dxdt&=\mathrm{PV}\int_0^{\frac{1}{n}+\epsilon}\int_{0}^{B}\Gamma(s-1+d+it)\frac{x^{d+it-1}}{1-n^2x^2}dtdx.
%\end{align}
Moreover, 
\begin{align*}
\lim_{B\to\infty}\mathrm{PV}\int_0^{\frac{1}{n}+\epsilon}\int_B^\infty \Gamma(s-1+d+it)\frac{x^{d+it-1}}{1-n^2x^2}dtdx=0,
\end{align*}
which follows from Stirling's formula \eqref{strivert}. This shows that the conditions mentioned in \cite[p.~94, Section 18]{hardypaper3} are satisfied. Thus we can interchange the order of integration in the first double integral in \eqref{justify}. This finally proves the validity of \eqref{int of int}.

Hence we can employ Parseval's formula \cite[p.~83, Equation (3.1.11)]{parsvel} for the functions in \eqref{c2} and \eqref{c1} which, for $0<c<2$, gives
\begin{align}\label{c3}
\frac{1}{2\pi i}\int_{(c)}\Gamma(s-1+z)\cot\left(\frac{\pi z}{2}\right)\left(\frac{4\pi^2 nk}{y}\right)^{-z}dz=\frac{2}{\pi}\mathrm{PV}\int_0^\infty e^{-\frac{4\pi^2 xk}{y}}\left(\frac{4\pi^2 xk}{y}\right)^{s-1}\frac{x}{x^2-n^2}dx.
\end{align}
Substituting the value from \eqref{c3} in \eqref{i2 almost2} so that
\begin{align}
I_2(y,s)=\Gamma(s-1)+\sum_{n=1}^\infty\mathrm{PV}\int_0^\infty e^{-\frac{4\pi^2 xk}{y}}\left(\frac{4\pi^2 xk}{y}\right)^{s-1}\frac{2x}{x^2-n^2}dx.\nonumber
\end{align}
By appealing to \eqref{inter of n and int} we can take the sum inside the integral in the above equation. Also note that $\Gamma(s-1)=\int_0^\infty e^{-\frac{4\pi^2 xk}{y}}\left(\frac{4\pi^2 xk}{y}\right)^{s-1}\frac{dx}{x}$ for Re$(s)>2, k>0$. Hence
\begin{align}\label{i2 evaluated}
I_2(y,s)&=\mathrm{PV}\int_0^\infty e^{-\frac{4\pi^2 xk}{y}}\left(\frac{4\pi^2 xk}{y}\right)^{s-1}\left\{\frac{1}{x}+ \sum_{n=1}^\infty\frac{2x}{x^2-n^2}\right\}dx\nonumber\\
&=\pi\ \mathrm{PV}\int_0^\infty\left(\frac{4\pi^2 xk}{y}\right)^{s-1}e^{-\frac{4\pi^2 xk}{y}}\cot(\pi x)dx,
\end{align}
which follows upon using \eqref{cotx}. The existence of the principal value integral appearing on the right-hand side of \eqref{i2 evaluated} is shown by Hardy \cite[p.~31]{hardy}. Substituting \eqref{apply residue} and \eqref{i2 evaluated} in \eqref{i1+i2} as well as in \eqref{sin i1+i2}, we get
\begin{align}\label{cosine ev}
\int_0^\infty\cos\left(\frac{2\pi kt}{y}\right)\frac{t^{s-1}}{e^t-1}dt&=\left(\frac{2\pi k}{y}\right)^{1-s}\left\{\cos\left(\frac{\pi s}{2}\right)\left(\pi\left(\frac{4\pi^2 k}{y}\right)^{s-1}\sum_{n=1}^\infty n^{s-1}e^{-\frac{4\pi^2 nk}{y}}-\frac{y\Gamma(s)}{4\pi k}\right)\right.\nonumber\\
&\quad\left.+\sin\left(\frac{\pi s}{2}\right)\pi\ \mathrm{PV}\int_0^\infty\left(\frac{4\pi^2 xk}{y}\right)^{s-1}e^{-\frac{4\pi^2 xk}{y}}\cot(\pi x)dx\right\},
\end{align}
and 
\begin{align}\label{sine ev}
\int_0^\infty\sin\left(\frac{2\pi kt}{y}\right)\frac{t^{s-1}}{e^t-1}dt&=\left(\frac{2\pi k}{y}\right)^{1-s}\left\{\sin\left(\frac{\pi s}{2}\right)\left(\pi\left(\frac{4\pi^2 k}{y}\right)^{s-1}\sum_{n=1}^\infty n^{s-1}e^{-\frac{4\pi^2 nk}{y}}-\frac{y\Gamma(s)}{4\pi k}\right)\right.\nonumber\\
&\quad\left.-\cos\left(\frac{\pi s}{2}\right)\pi\ \mathrm{PV}\int_0^\infty\left(\frac{4\pi^2 xk}{y}\right)^{s-1}e^{-\frac{4\pi^2 xk}{y}}\cot(\pi x)dx\right\}.
\end{align}
Substituting \eqref{cosine ev} and \eqref{sine ev} in \eqref{integral} and simplifying, we are led to
\begin{align}\label{minus sign}
\sum_{n=1}^\infty\frac{(n-a)^{s-1}}{e^{(n-a)y}-1}&=\frac{\Gamma(s)\zeta(s)}{y^s}+\left(\frac{2\pi}{y}\right)^s\cos\left(\frac{\pi s}{2}\right)\sum_{n=1}^\infty n^{s-1}\sum_{k=1}^\infty\cos\left(2\pi ak\right) e^{-\frac{4\pi^2 nk}{y}}-\frac{\Gamma(s)}{(2\pi)^{s}}\cos\left(\frac{\pi s}{2}\right)\sum_{k=1}^\infty\frac{\cos\left(2\pi ak\right)}{k^s}\nonumber\\
&\quad-\left(\frac{2\pi}{y}\right)^s\sin\left(\frac{\pi s}{2}\right)\sum_{n=1}^\infty n^{s-1}\sum_{k=1}^\infty\sin\left(2\pi ak\right) e^{-\frac{4\pi^2 nk}{y}}+\frac{\Gamma(s)}{(2\pi)^{s}}\sin\left(\frac{\pi s}{2}\right)\sum_{k=1}^\infty\frac{\sin\left(2\pi ak\right)}{k^s}\nonumber\\
&\quad+\left(\frac{2\pi}{y}\right)^s\sin\left(\frac{\pi s}{2}\right)\sum_{k=1}^\infty\cos(2\pi ak) \mathrm{PV}\int_0^\infty x^{s-1}e^{-\frac{4\pi^2 xk}{y}}\cot(\pi x)dx\nonumber\\
&\quad+\left(\frac{2\pi}{y}\right)^s\cos\left(\frac{\pi s}{2}\right)\sum_{k=1}^\infty\sin(2\pi ak) \mathrm{PV}\int_0^\infty x^{s-1}e^{-\frac{4\pi^2 xk}{y}}\cot(\pi x)dx.
\end{align}
Invoking Lemma \ref{intechange of PV and summation} we can interchange the summation and integration for the last expression on the right-hand side of \eqref{minus sign}. Also note that one can justify the same for the series involving $\cos(2\pi a k)$. Hence, after simplification, \eqref{minus sign} becomes
\begin{align}
\sum_{n=1}^\infty\frac{(n-a)^{s-1}}{e^{(n-a)y}-1}&=\frac{\Gamma(s)\zeta(s)}{y^s}+\left(\frac{2\pi}{y}\right)^s\sum_{n=1}^\infty n^{s-1}\sum_{k=1}^\infty\cos\left(\frac{\pi s}{2}+2\pi ak\right) e^{-4\pi^2 nk/y}-\frac{\Gamma(s)}{(2\pi)^{s}}\sum_{k=1}^\infty\frac{\cos\left(\frac{\pi s}{2}+2\pi ak\right)}{k^s}\nonumber\\
&\quad+\left(\frac{2\pi}{y}\right)^s \mathrm{PV}\int_0^\infty \left(\sum_{k=1}^\infty\sin\left(\frac{\pi s}{2}+2\pi ak\right)e^{-4\pi^2 kx/y}\right) x^{s-1}\cot(\pi x)dx.\nonumber
\end{align}
Making the change of variable $x\to xy/(2\pi)$ in the integral and rearranging the terms in the above expression, we get
\begin{align}\label{alt1}
&\frac{\Gamma(s)\zeta(s)}{y^s}+\left(\frac{2\pi}{y}\right)^s\sum_{n=1}^\infty n^{s-1}\sum_{k=1}^\infty\cos\left(\frac{\pi s}{2}+2\pi ak\right)e^{-4\pi^2 nk/y}\nonumber\\
&=\frac{\Gamma(s)}{(2\pi)^{s}}\sum_{k=1}^\infty\frac{\cos\left(\frac{\pi s}{2}+2\pi ak\right)}{k^s}+\sum_{n=1}^\infty\frac{(n-a)^{s-1}}{e^{(n-a)y}-1}-\mathrm{PV}\int_0^\infty\left(\sum_{k=1}^\infty\sin\left(\frac{\pi s}{2}+2\pi ak\right)e^{-2\pi kx}\right) x^{s-1} \cot\left(\frac{1}{2}y x\right)dx.
\end{align}
Now, using the fact $\cos(\theta)=(e^{i\theta}+e^{-i\theta})/2$, we find 
\begin{align}\label{cos}
\sum_{k=1}^\infty\cos\left(\frac{\pi s}{2}+2\pi ak\right)e^{-4\pi^2 nk/y}&=\frac{1}{2}e^{\pi is/2}\sum_{k=1}^\infty e^{-\left(\frac{4\pi^2 n}{y}-2\pi ia\right)k}+\frac{1}{2}e^{-\pi is/2}\sum_{k=1}^\infty e^{-\left(\frac{4\pi^2 n}{y}+2\pi ia\right)k}\nonumber\\
&=\frac{1}{2}\left(\frac{e^{\pi is/2}}{e^{\frac{4\pi^2 n}{y}-2\pi ia}-1}+\frac{e^{-\pi is/2}}{e^{\frac{4\pi^2 n}{y}+2\pi ia}-1}\right).
\end{align}
Similarly, 
\begin{align}\label{sin}
\sum_{k=1}^\infty\sin\left(\frac{\pi s}{2}+2\pi ak\right)e^{-2\pi kx}=\frac{1}{2i}\left(\frac{e^{\pi is/2}}{e^{2\pi x-2\pi ia}-1}-\frac{e^{-\pi is/2}}{e^{2\pi x+2\pi ia}-1}\right).
\end{align}
Substituting  \eqref{cos} and \eqref{sin} in \eqref{alt1}, we get
\begin{align*}
&\frac{\Gamma(s)\zeta(s)}{y^s}+\left(\frac{2\pi}{y}\right)^s\frac{1}{2}\sum_{n=1}^\infty n^{s-1}\left(\frac{e^{\pi is/2}}{e^{\frac{4\pi^2 n}{y}-2\pi ia}-1}+\frac{e^{-\pi is/2}}{e^{\frac{4\pi^2 n}{y}+2\pi ia}-1}\right)\nonumber\\
&=\frac{\Gamma(s)}{(2\pi)^{s}}\sum_{k=1}^\infty\frac{\cos\left(\frac{\pi s}{2}+2\pi ak\right)}{k^s}+\sum_{n=1}^\infty\frac{(n-a)^{s-1}}{e^{(n-a)y}-1}\\
&\quad-\frac{1}{2i}\mathrm{PV}\int_0^\infty x^{s-1}\left(\frac{e^{\pi is/2}}{e^{2\pi x-2\pi ia}-1}-\frac{e^{-\pi is/2}}{e^{2\pi x+2\pi ia}-1}\right) \cot\left(\frac{1}{2}y x\right)dx.
\end{align*}
Finally, we arrive at \eqref{ram with a eqn} after multiplying by $\left(2\pi/y\right)^{-s}$ on the both sides of the above equation and then letting $4\pi^2/y=\alpha$ with $\alpha\beta=4\pi^2$.
%It remains to justify the interchange of the order of the summation and integration. We do it by using \cite[p.~30, Theorem 21.]{temme}. Note that the series $\sum_{k=1}^\infty\sin\left(\frac{\pi s}{2}+2\pi ak\right) e^{-4\pi^2 kx/y}$ is uniformly convergent in $x$ on any compact interval of $(0,\infty)$. Also by using \eqref{bound} and \eqref{sin}, it is easy to see that 
%\begin{align*}
%\int_0^\infty\left|\sum_{k=1}^\infty\sin\left(\frac{\pi s}{2}+2\pi ak\right) e^{-4\pi^2 kx/y}\cot(\pi x) x^{s-1}\right|dx&=\int_0^\infty\left|\frac{1}{2i}\left(\frac{e^{\pi is/2}}{e^{2\pi x-2\pi ia}-1}-\frac{e^{-\pi is/2}}{e^{2\pi x+2\pi ia}-1}\right)\cot(\pi x)x^{s-1}\right|dx\\
%&<\infty,
%\end{align*}
%for Re$(s)>2$ when $a=0$, and for Re$(s)>1$ when $0<a<1$. Hence the hypothesises of \cite[p.~30, Theorem 21.]{temme} are satisfied. Therefore, it justify the interchange of the summation and integration.
\end{proof}
	
\begin{proof}[Corollary \textup{\ref{a half s 2m}}][]
Let $a=1/2$ and $s=2m, m\in\mathbb{N}$ in Theorem \ref{ram with a} and observe that the principal value integral vanishes. The result then follows upon using Euler's formula \cite[p.~5, Equation (1.14)]{temme}
\begin{align}\label{euler formula}	
\zeta(2 m ) = (-1)^{m +1} \frac{(2\pi)^{2 m}B_{2 m }}{2 (2 m)!}.
\end{align} 
\end{proof}

\begin{proof}[Corollary \textup{\ref{close form}}][]
%For an even integer $m$, we have \cite[p.~23]{dg} (also see \cite[p.~25, Exercise 15(c)]{apostal})
%\begin{align*}
%\sum_{n=1}^\infty\frac{(2n-1)^{2m+1}}{e^{\pi (2n-1)}+1}=(2^{2m+1}-1)\frac{B_{2m+2}}{4m+4}.
%\end{align*}
Let $\alpha=\beta=2\pi$ and $m$ be an odd positive integer in Corollary \ref{a half s 2m} so as to get
\begin{align}\label{3rdcor}
	\sum_{n=1}^\infty\frac{n^{2m-1}}{e^{2n\pi}+1}-\frac{1}{2^{2m-1}}\sum_{n=1}^\infty\frac{(2n-1)^{2m-1}}{e^{(2n-1)\pi}-1}=-2^{1-2m}\frac{B_{2m}}{4m}.
\end{align}
Now use the fact
\begin{align}\label{4thcor}
\sum_{n=1}^\infty\frac{(2n-1)^{2m-1}}{e^{(2n-1)\pi}-1}=\sum_{\substack{n=1\\n\hspace{0.5mm}\textup{odd}}}^\infty\frac{n^{2m-1}}{e^{n\pi}-1}	=	\sum_{n=1}^\infty\frac{n^{2m-1}}{e^{n\pi}-1}-2^{2m-1}\sum_{n=1}^\infty\frac{n^{2m-1}}{e^{2n\pi}-1},
\end{align}
and Glaisher's evaluation \cite{glaisher}
\begin{equation}\label{glaisher_eqn}
	\sum_{n=1}^{\infty}\frac{n^{2m-1}}{e^{2n\pi}-1}=\frac{B_{2m}}{4m} \hspace{5mm}(m\hspace{1mm}\textup{odd}>1)
\end{equation}
to arrive at \eqref{1stcor}.

As far as the proof of \eqref{2ndcor} is concerned, we let $m=1$ in \eqref{3rdcor} and use \eqref{4thcor} and use Schl\"{o}milch's result \cite{schlomilch}
\begin{equation*}
	\sum_{n=1}^{\infty}\frac{n}{e^{2n\pi}-1}=\frac{1}{24}-\frac{1}{8\pi}.
\end{equation*}
\end{proof}

\begin{proof}[Corollary \textup{\ref{a=1/4and 3/4}}][]
Letting $s=2m$ and $a=1/4$ in Theorem \ref{ram with a} and simplifying, we get
\begin{align}\label{a1/4}
&\alpha^m\left\{\frac{\Gamma(2m)\zeta(2m)}{(2\pi)^{2m}}+(-1)^{m+1}\sum_{n=1}^\infty \frac{n^{2m-1}}{e^{2n\alpha}+1}\right\}\nonumber\\
&=\beta^m\left\{\frac{(-1)^m\Gamma(2m)}{(2\pi)^{2m}}\sum_{k=1}^\infty\frac{\cos(\pi k/2)}{k^{2m}}+\sum_{n=1}^\infty\frac{(n-1/4)^{2m-1}}{e^{(n-1/4)\beta}-1}-\frac{(-1)^m}{2}\textup{PV}\int_0^\infty \frac{x^{2m-1}\cot(\beta x/2)}{\cosh(2\pi x)}dx\right\}.
\end{align}
Now take $s=2m$ and $a=3/4$ in Theorem \ref{ram with a} to obtain
\begin{align}\label{a3/4}
&\alpha^m\left\{\frac{\Gamma(2m)\zeta(2m)}{(2\pi)^{2m}}+(-1)^{m+1}\sum_{n=1}^\infty \frac{n^{2m-1}}{e^{2n\alpha}+1}\right\}\nonumber\\
&=\beta^m\left\{\frac{(-1)^m\Gamma(2m)}{(2\pi)^{2m}}\sum_{k=1}^\infty\frac{(-1)^k\cos(\pi k/2)}{k^{2m}}+\sum_{n=1}^\infty\frac{(n-3/4)^{2m-1}}{e^{(n-3/4)\beta}-1}+\frac{(-1)^m}{2}\textup{PV}\int_0^\infty \frac{x^{2m-1}\cot(\beta x/2)}{\cosh(2\pi x)}dx\right\}.
\end{align}
Now add \eqref{a1/4} and \eqref{a3/4} so that
\begin{align}
&\alpha^m\left\{2\frac{\Gamma(2m)\zeta(2m)}{(2\pi)^{2m}}+2(-1)^{m+1}\sum_{n=1}^\infty \frac{n^{2m-1}}{e^{2n\alpha}+1}\right\}\nonumber\\
&=\beta^m\left\{\frac{(-1)^m\Gamma(2m)}{(2\pi)^{2m}}\sum_{k=1}^\infty\frac{(1+(-1)^k)\cos(\pi k/2)}{k^{2m}}+\sum_{n=1}^\infty\frac{(n-1/4)^{2m-1}}{e^{(n-1/4)\beta}-1}+\sum_{n=1}^\infty\frac{(n-3/4)^{2m-1}}{e^{(n-3/4)\beta}-1}\right\}.\nonumber
\end{align}
Using the fact $\sum_{n=1}^\infty(-1)^k/k^{2m}=2^{-2m}(2-2^{2m})\zeta(2m)$ in the above equation, we arrive at \eqref{a=1/4and 3/4 eqn}.

Next, subtracting \eqref{a3/4} from \eqref{a1/4} yields
\begin{align}
&\frac{(-1)^m\Gamma(2m)}{(2\pi)^{2m}}\sum_{k=1}^\infty\frac{(1-(-1)^k)\cos(\pi k/2)}{k^{2m}}+\sum_{n=1}^\infty\frac{(n-1/4)^{2m-1}}{e^{(n-1/4)\beta}-1}-\sum_{n=1}^\infty\frac{(n-3/4)^{2m-1}}{e^{(n-3/4)\beta}-1}\nonumber\\
&=(-1)^m\textup{PV}\int_0^\infty \frac{x^{2m-1}\cot(\beta x/2)}{\cosh(2\pi x)}dx.\nonumber
\end{align}
Note that the first sum on the left-hand side of the above equation vanishes. Now replacing $\beta$ by $4\beta$ and rewriting the left-hand side in terms of the Dirichlet character $\chi$ defined in \eqref{character}, we are led to \eqref{a=1/4and 3/4 eqn1}. 
\end{proof}

\section{A simple proof of the transformation for $\sum_{n=1}^\infty \sigma_{2m}(n)e^{-ny}$}\label{appl}

In \cite{dkk1}, this theorem was obtained for the first time as a corollary of a more general result, namely, \eqref{maineqn}. Hence the absolute convergence of the series on the right-hand side of \eqref{dkka=2m eqn} resulted automatically. In what follows, we not only give a direct proof of this result, but also prove from scratch the convergence of the series.

To that end, we first prove the identity for $y>0$ and later extend it to Re$(y)>0$ by analytic continuation. We begin by showing the absolute convergence of the series on the right-hand side of \eqref{dkka=2m eqn}. Note that for $w>0$, \eqref{symmetry} and Theorem \ref{minusonelemma} imply
\begin{align}\label{before headache}
\sinh(w)\mathrm{Shi}(w)-\cosh(w)\mathrm{Chi}(w)+\sum_{j=1}^m(2j-1)!w^{-2j}=(-1)^m\mathfrak{R}_m(1,w)=\frac{(-1)^m}{w^{2m}}\mathfrak{R}_m(w,1).
\end{align}
Now employ Lemma \ref{genraabeasyinfty}  for $\mathfrak{R}_m(w,1)$, and then let $w=4\pi^2n/y$, where $y>0$ (as assumed at the beginning of the proof), so that as $n\to\infty$, we have
\begin{equation}\label{headache}
	\sinh\left(\frac{4\pi^2n}{y}\right)\mathrm{Shi}\left(\frac{4\pi^2n}{y}\right)-\cosh\left(\frac{4\pi^2n}{y}\right)\mathrm{Chi}\left(\frac{4\pi^2n}{y}\right)+\sum_{j=1}^m(2j-1)!\left(\frac{4\pi^2n}{y}\right)^{-2j}=O_{m, y}\left(\frac{1}{n^{2m+2}}\right).
\end{equation}
The absolute convergence of  $\sum_{n=1}^{\infty}\sigma_{2m}(n)/n^{2m+2}$ then implies that of the series on the right-hand side of \eqref{dkka=2m eqn} with the help of the above estimate. 

We now prove \eqref{dkka=2m eqn}. Let $a=0$ and $s=2m+1$ in \eqref{4.3} so that
\begin{align}
\sum_{n=1}^\infty \sigma_{2m}(n)e^{-ny}=\frac{(2m)!}{y^{2m+1}}\left\{\zeta(2m+1)+\sum_{k=1}^\infty\left(\zeta\left(1+2m,1+\frac{2\pi ik}{y}\right)+\zeta\left(1+2m,1-\frac{2\pi ik}{y}\right)\right)\right\}.\nonumber
\end{align}
Using \eqref{zeta shift}, we can see that for $m\in\mathbb{N}$,
 $$\zeta\left(1+2m,1+\frac{2\pi ik}{y}\right)+\zeta\left(1+2m,1-\frac{2\pi ik}{y}\right)=\zeta\left(1+2m,\frac{2\pi ik}{y}\right)+\zeta\left(1+2m,-\frac{2\pi ik}{y}\right).$$
%\begin{align*}
%&\zeta\left(1+2m,1+\frac{2\pi ik}{y}\right)+\zeta\left(1+2m,1-\frac{2\pi ik}{y}\right)\\
%&=\zeta\left(1+2m,\frac{2\pi ik}{y}\right)+\zeta\left(1+2m,-\frac{2\pi ik}{y}\right)-\left(\frac{2\pi k}{y}\right)^{-(2m+1)}\left(i^{-(2m+1)}+(-i)^{-(2m+1)}\right)\\
%&=\zeta\left(1+2m,\frac{2\pi ik}{y}\right)+\zeta\left(1+2m,-\frac{2\pi ik}{y}\right)-\left(\frac{2\pi k}{y}\right)^{-(2m+1)}\cos\left((2m+1)\frac{\pi}{2}\right)\\
%&=\zeta\left(1+2m,\frac{2\pi ik}{y}\right)+\zeta\left(1+2m,-\frac{2\pi ik}{y}\right).
%\end{align*}
Now employ Theorem \ref{nagoyagen} with $z=m$ and $w=2\pi k/y$ in the above equation to see that
\begin{align}
\sum_{n=1}^\infty \sigma_{2m}(n)e^{-ny}&=\frac{(2m)!}{y^{2m+1}}\Bigg\{\zeta(2m+1)+\frac{\cos(\pi m)}{m}\left(\frac{2\pi}{y}\right)^{-2m}\sum_{k=1}^\infty\frac{1}{k^{2m}}\nonumber\\
&\quad+2\sum_{k=1}^\infty\sum_{n=1}^\infty\int_0^\infty\left(\frac{1}{(v-2\pi ik/y)^{2m+1}}+\frac{1}{(v+2\pi ik/y)^{2m+1}}\right)\cos(2\pi nv)\ dv\Bigg\}\nonumber\\
&=\frac{(2m)!}{y^{2m+1}}\Bigg\{\zeta(2m+1)+\frac{(-1)^m}{m}\left(\frac{2\pi}{y}\right)^{-2m}\zeta(2m)+2(2\pi)^{2m}\sum_{k=1}^\infty\sum_{n=1}^\infty n^{2m}\nonumber\\
&\quad\times\int_0^\infty\left(\frac{1}{(t-4\pi^2 ink/y)^{2m+1}}+\frac{1}{(t+4\pi^2 ink/y)^{2m+1}}\right)\cos(t)\ dt\Bigg\},\nonumber
\end{align}
where in the last step we made change of variable $v=t/(2\pi n)$. Letting $nk=\ell$, we see that
\begin{align}\label{dk1}
\sum_{n=1}^\infty \sigma_{2m}(n)e^{-ny}&=\frac{(2m)!}{y^{2m+1}}\Bigg\{\zeta(2m+1)+\frac{(-1)^m}{m}\left(\frac{2\pi}{y}\right)^{-2m}\zeta(2m)+2(2\pi)^{2m}\sum_{\ell=1}^\infty\sum_{n|\ell} n^{2m}\nonumber\\
&\quad\times\int_0^\infty\left(\frac{1}{(t-4\pi^2 i\ell/y)^{2m+1}}+\frac{1}{(t+4\pi^2 i\ell/y)^{2m+1}}\right)\cos(t)\ dt\Bigg\}.
\end{align}
Next, invoke Theorem \ref{minusonelemma} with $w=4\pi^2\ell/y$ in \eqref{dk1} to arrive at
\begin{align}\label{beforezeta2m}
\sum_{n=1}^\infty \sigma_{2m}(n)e^{-ny}&=\frac{(2m)!}{y^{2m+1}}\Bigg\{\zeta(2m+1)+\frac{(-1)^m}{m}\left(\frac{2\pi}{y}\right)^{-2m}\zeta(2m)+\frac{4(-1)^{m}(2\pi)^{2m}}{(2m)!}\sum_{\ell=1}^\infty\sigma_{2m}(\ell)\nonumber\\
&\quad\times\Bigg\{\sinh\left(\frac{4\pi^2\ell}{y}\right)\mathrm{Shi}\left(\frac{4\pi^2\ell}{y}\right)-\cosh\left(\frac{4\pi^2\ell}{y}\right)\mathrm{Chi}\left(\frac{4\pi^2\ell}{y}\right)+\sum_{j=1}^m(2j-1)!\left(\frac{4\pi^2\ell}{y}\right)^{-2j}\Bigg\}.
\end{align}
%We have the following well-known formula \cite[p.~5, Equation (1.14)]{temme}
%\begin{align}\label{zeta(2m)}
%\zeta(2m)=\frac{(-1)^{m+1}(2\pi)^{2m}B_{2m}}{2(2m)!}.
%\end{align}
Using \eqref{euler formula} in \eqref{beforezeta2m} and rearranging the terms leads to \eqref{dkka=2m eqn} for $y>0$. The result can be extended by analytic continuation to Re$(y)>0$. This is seen as follows. Clearly, the left-hand side of \eqref{dkka=2m eqn}  is analytic in this region. We now show that the series on the right is also analytic. In order to prove this using Weierstrass' theorem on analytic functions, we need only show that  \eqref{headache} holds for Re$(y)>0$ as well. To that end, employing $(1-\sqrt{\xi})^{-(2m+1)}-(1+\sqrt{\xi})^{-(2m+1)}=2(2m+1)\sqrt{\xi}\ {}_2F_1\left(m+1,m+\frac{3}{2};\frac{3}{2};\xi\right)$, we find that
\begin{align}\label{headache solved}
\int_0^\infty\left(\frac{1}{(t-iw)^{2m+1}}+\frac{1}{(t+iw)^{2m+1}}\right)\cos(t) dt&=\frac{2(2m+1)}{(-1)^mw^{2m+2}}\int_{0}^{\infty}t\cos(t){}_2F_{1}\left(m+1,m+\frac{3}{2};\frac{3}{2};-\frac{t^2}{w^2}\right)dt.
\end{align}
The integral on the right can be evaluated in terms of the Meijer $G$-function $\MeijerG*{3}{1}{2}{4}{1,\frac{3}{2}}{1,m+1,m+\frac{3}{2},\frac{3}{2}}{\frac{w^2}{4}}$ employing \cite[p.~81, Formula \textbf{8.17.6}]{tit2}. That
\begin{align}\label{G bound}
\MeijerG*{3}{1}{2}{4}{1,\frac{3}{2}}{1,m+1,m+\frac{3}{2},\frac{3}{2}}{\frac{w^2}{4}}=-\frac{w^2}{\sqrt{\pi}2^{2m+1}}\sum_{j=1}^{r+1}\frac{\Gamma(2m+2j)}{w^{2j}}+O(w^{-2r-4}),\quad \textup{as}\ w\to\infty,\ \textup{Re}(w)>0
\end{align}
can then be obtained using the asymptotic of this Meijer $G$-function given in \cite[p.~179, Theorem 2]{Luke}. With $w=4\pi^2n/y, \text{Re}(y)>0$, the first equality of \eqref{before headache} and \eqref{headache solved} finally prove \eqref{headache}. This completes the proof of \eqref{dkka=2m eqn} for Re$(y)>0$.

\qed

\section{Asymptotics of the plane partitions generating function}\label{asymptotic}
\begin{proof}[Corollary \textup{\ref{asym for general m}}][]
The following estimate can be directly shown to hold as $y\to0$ in Re$(y)>0$ as is done later. However, we first prove it separately for real $y\to0^+$ owing to its simplicity. 

Indeed, for real $y\to0^+$, Lemma \ref{genraabeasyinfty} along with \eqref{before headache} imply that
\begin{align}\label{summand bound}
&\sinh\left(\frac{4\pi^2n}{y}\right)\mathrm{Shi}\left(\frac{4\pi^2n}{y}\right)-\cosh\left(\frac{4\pi^2n}{y}\right)\mathrm{Chi}\left(\frac{4\pi^2n}{y}\right)+\sum_{j=1}^m(2j-1)!\left(\frac{4\pi^2n}{y}\right)^{-2j}\nonumber\\
&=\frac{-1}{(4\pi^2n)^{2m}}y^{2m}\sum_{j=1}^{r+1}\frac{\Gamma(2m+2j)}{(4\pi^2n)^{2j}}y^{2j}+O\left(\frac{y^{2r+2m+4}}{n^{2r+2m+4}}\right).
\end{align}
For complex $y$ in Re$(y)>0$ such that $y\to0$, \eqref{summand bound} is seen to hold from the first equality of \eqref{before headache}, \eqref{headache solved} and \eqref{G bound}.

Substituting \eqref{summand bound} in \eqref{dkka=2m eqn}, we deduce that
\begin{align}
\sum_{n=1}^\infty\sigma_{2m}(n)e^{-ny}&=\frac{(2m)!}{y^{2m+1}}\zeta(2m+1)-\frac{B_{2m}}{2my}-(-1)^m\frac{2}{\pi}\left(\frac{2\pi}{y}\right)^{2m+1}\frac{y^{2m}}{(2\pi)^{4m}}\sum_{n=1}^\infty\frac{\sigma_{2m}(n)}{n^{2m}}\sum_{j=1}^{r+1}\frac{\Gamma(2m+2j)}{(4\pi^2n)^{2j}}y^{2j}\nonumber\\
&\qquad+O\left(y^{2r+3}\right)\nonumber\\
&=\frac{(2m)!}{y^{2m+1}}\zeta(2m+1)-\frac{B_{2m}}{2my}-(-1)^m\frac{2}{y\pi(2\pi)^{2m-1}}\sum_{j=1}^{r+1}\frac{\Gamma(2m+2j)}{(4\pi^2)^{2j}}y^{2j}\sum_{n=1}^\infty\frac{\sigma_{2m}(n)}{n^{2m+2j}}+O\left(y^{2r+3}\right)\nonumber.
\end{align}
Using the well-known identity $\sum_{n=1}^\infty\sigma_a(n)n^{-s}=\zeta(s)\zeta(s-a)$, where $\textup{Re}(s)>1, \textup{Re}(s-a)>1,$ in the above expression, we arrive at \eqref{asym for general m eqn}. 
\end{proof}

We are now ready to derive Wright's result from \cite{wright} as a special case of Corollary \ref{asym for general m}.
\begin{proof}[Corollary \textup{\ref{wright asym}}][]
Letting $m=1$ and $y=\log(1/x), |x|<1,$ in Corollary \ref{asym for general m} and using \eqref{plain partition}, as $x\to1^-$, we have
\begin{align}
x\frac{d}{dx}\log F(x)&=-\frac{2\zeta(3)}{(\log x)^3}+\frac{1}{12\log x}-\frac{1}{\pi^2}\sum_{j=1}^{r+1}\frac{\Gamma(2j+2)\zeta(2j+2)\zeta(2j)}{(2\pi)^{4j}}(\log x)^{2j-1}+O\left(-(\log x)^{2r+3}\right).
\end{align} 
Now divide both sides by $x$ and then integrate with respect to $x$ to get
\begin{align}
\log F(x)=c+\frac{\zeta(3)}{(\log x)^2}+\frac{1}{12}\log\log x-\frac{1}{\pi^2}\sum_{j=1}^{r+1}\frac{\Gamma(2j+2)\zeta(2j+2)\zeta(2j)}{2j(2\pi)^{4j}}(\log x)^{2j}+O\left((\log x)^{2r+4}\right),
\end{align}
where $c$ is an integrating constant. Exponentiating both sides of the above equation, we arrive at \eqref{wright asym eqn}. 
\end{proof}

\section{Concluding remarks}\label{cr}

For general $a$ with $0\leq a<1$, the generalized Lambert series 
\begin{equation*}
	\sum_{n=1}^{\infty}\frac{(n-a)^{s-1}}{e^{(n-a)z}-1}\hspace{8mm}(s\in\mathbb{C}, \textup{Re}(z)>0)
\end{equation*}
does not seem to have been studied before. It makes its appearance for the first time in Theorem \ref{ram with a} of our paper. It may be interesting to undertake a further study of this series.

In \cite[Theorem 1]{bradley2002}, Bradley obtains a generalization of Ramanujan's formula \eqref{rameqn} for periodic functions $g$ with period $m\in\mathbb{N}$. When $g$ is even, his transformation involves the series of the type $\displaystyle\sum_{n=1}^\infty\frac{g(n)n^{-2m-1}}{e^{n\beta}-1}$ whereas for $g$ odd, it involves $\displaystyle\sum_{n=1}^\infty\frac{g(n)n^{-2m}}{e^{n\beta}-1}$. Observe that the series in our \eqref{a=1/4and 3/4 eqn1} involves $\displaystyle\sum_{n=1}^\infty\frac{\chi(n)n^{2m-1}}{e^{n\beta}-1}$, where $\chi(n)$ defined in \eqref{character} is an odd Dirichlet character and $m\in\mathbb{N}, m>1$, and hence does not fall under the purview of Bradley's transformation. Thus it may be worthwhile to see if a more general transformation encompassing our series exists. We note that another series which is not covered by Bradley's transformation is $\displaystyle\sum_{n=1}^\infty\frac{n^{-2m}}{e^{n\beta}-1}$, for which a transformation was recently obtained in \cite[Theorem 2.12]{dkk1}.

In \cite{dkk1}, \eqref{dkka=2m eqn} was obtained as a special case of \eqref{maineqn} by tour de force whereas in the present paper, this has been accomplished directly. One can then ask if a direct proof of Theorem 2.12  of \cite{dkk1}, which is a transformation for $\displaystyle\sum_{n=1}^\infty\frac{n^{-2m}}{e^{n\beta}-1}$, can be derived without resorting to \eqref{maineqn}.

\section*{\textbf{Acknowledgements}}
The authors sincerely thank George E. Andrews for a helpful discussion on plane partitions. They would also like to thank  Donghun Yu from POSTECH library for arranging references \cite{hardy, hardypaper2, wigert} for us. The first author's research was partially supported by the Swarnajayanti Fellowship Grant SB/SJF/2021-22/08 of SERB (Govt. of India) and the CRG grant CRG/2020/002367 of SERB. The second author's research was supported by the grant IBS-R003-D1 of the IBS-CGP, POSTECH, South Korea. Both the authors sincerely thank the respective funding agencies for their support. Part of this work was done when the second author was visiting IIT Gandhinagar in May 2022. He sincerely thanks IBS-CGP for the financial support and IIT Gandhinagar for its hospitality.

\end{document}